\documentclass[times,10pt]{amsart}

\newcommand{\N}{\ensuremath{\mathbb{N}}}
\newcommand{\Z}{\ensuremath{\mathbb{Z}}}

\DeclareMathOperator{\dd}{\mathrm{d}}
\newcommand{\F}{\ensuremath{\mathcal{F}}}
\newcommand{\I}{\ensuremath{\mathcal{I}}}

\newcommand{\n}{\ensuremath{\infty}}
\newcommand{\resp}{{\it resp. }}
\renewcommand{\phi}{\varphi}
\renewcommand{\cong}{\simeq}
\renewcommand{\resp}{{\it resp. }}

\usepackage{color}
\usepackage{xcolor}
 
 \newtheorem{thm}{Theorem}[section]
 \newtheorem{defn}[thm]{Definition}%[section]
 \newtheorem{prop}[thm]{Proposition}%[section]
 \newtheorem{lemma}[thm]{Lemma}%[section]
 \newtheorem{rem}[thm]{Remark}%[section]

\date{\today}

\title{Kato-Milne cohomology group over rational function fields in characteristic  2, II}
\author{Ahmed Laghribi}
\author{Trisha Maiti}

\address{Univ. Artois, UR 2462, Laboratoire de Math{\'e}matiques de Lens (LML), F-62300 Lens, France}

\email{ahmed.laghribi@univ-artois.fr, trisha.maiti@univ-artois.fr}

\date{\today}

\begin{document}

\maketitle
\begin{abstract} 
Our aim in this paper is to prove in the setting of Kato-Milne cohomology in characteristic $2$ an exact sequence which is analogue to  the Milnor-Scharlau  sequence 
\cite[Theorem 6.2]{Sc}.  This is an extension of the Milnor exact sequence proved in \cite{ex}. 
\end{abstract}

\tableofcontents

\begin{sloppypar}
\section{Introduction}\label{S1}

In \cite[Thoerem 6.2]{AJ} Aravire and Jacob  determined  the Witt group of  nonsingular quadratic forms $W_q( \F(x))$, where $\F$ is a field of characteristic $2$ and $x$ is an indeterminate over $\F$. To this end, they proved the analogue of the Milnor-Scharlau exact sequence \cite[Theorem 6.2]{Sc} known in characteristic not $2$. Since nonsingular quadratic forms in characteristic $2$ are related to differential forms by Kato's isomorphism (\ref{Ke}), it is natural to ask for an analogue of the Milnor-Scharlau exact sequence in the setting of Kato-Milne cohomology. This is what we will accomplish in this paper by following the same process developed in \cite{AJ}. In \cite{ex} we already proved, in the setting of Kato-Milne cohomology, the analogue of the Milnor exact sequence \cite{Mi}. So in this paper, we extend that short exact sequence using the transfers described in Definition \ref {dt}.  To state our main result, we need to recall some backgrounds on quadratic forms, differential forms and Teichm\"uller lifting.

Let $F$ be a field of characteristic $2$.  Up to isometry, any quadratic form $\phi$ over $F$  can be written up to isometry as follows:\[\phi\cong [a_1,b_1]\perp \cdots \perp [a_r, b_r]\perp \langle c_1\rangle\perp \cdots \perp \langle c_s\rangle,\] where $[a,b]$ and $\langle c\rangle$ denote the quadratic forms $aX^2+XY+bY^2$ and $\langle c\rangle$, respectively. The quadratic form  $\phi$ is said to be nonsingular if $s=0$.
 
For $a_1, \cdots, a_n\in F^*=:F\setminus\{0\}$, let $\left<a_1, \cdots, a_n\right>_b$ denote the diagonal bilinear form given as follows:$$((x_1, \cdots, x_n), (y_1, \cdots, y_n))\mapsto \sum_{i=1}^na_ix_iy_i.$$ 

Let $W(F)$ be the Witt ring of nondegenerate symmetric bilinear forms over $F$, and $W_q(F)$ the Witt group of nonsingular quadratic forms over $F$.  It is known that $W_q(F)$ is endowed with a $W(F)$-module structure induced by tensor product \cite{El}.  Let $I(F)$ be the fundamental ideal of $W(F)$. For $m\in \N_0$, let $I^m(F)$ be the $m$-th power of $I(F)$. This ideal is additively generated by the $m$-fold bilinear  Pfister forms $ \langle\langle a_1,\ldots,a_m\rangle\rangle_b:= \langle 1,a_1\rangle_b\otimes\ldots\otimes \langle 1,a_m\rangle_b$ for $a_i\in F^*$ (we take $I^0(F)=W(F)$).  Similarly, the subgroup $ I_q^{m+1}(F)=I^m(F)\otimes W_q(F)$ is additively generated by the $(m+1)$-fold quadratic Pfister forms $\langle\langle a_1,\ldots,a_m;b]]:=\langle \langle a_1,\ldots,a_m\rangle\rangle_b\otimes[1,a]$, for $a_i\in F^*$ and $b\in F$. Note that $I_q^1(F)= W_q(F)$.  The notation $\overline{I}_q^{m+1}(F)$ means the quotient $I_q^{m+1}(F)/I_q^{m+2}(F)$.

For any integer $m\geq 1$, let $\Omega_F^m$ be the space of $m$-differential forms over $F$. We also take $\Omega_F^0=F$ and $\Omega_F^m=0$ for $m<0$.  Let  $ H_2^{m+1}(F)$ be the  Kato-Milne cohomology group of degree $m$, defined by \[ H_2^{m+1}(F)= \Omega_F^m/(\dd \Omega_F^{m-1}+ \wp(\Omega_F^m)),\] where $\dd$ and $\wp$ denote the differential and the Artin-Schreier operators, respectively. We refer to \cite[Introduction]{ex} for definitions. For each  $m\in \N$, we have Kato's isomorphism $f(F)_{m+1}: H_2^{m+1}(F)\xrightarrow{} \overline{I}_q^{m+1}(F)$ defined by 
  \begin{equation}\label{Ke}
      f(F)_{m+1}\left(\overline{x\frac{\dd x_1}{x_1}\wedge \ldots\wedge \frac{\dd x_m}{x_m}}\right)=\overline{\langle\langle x_1,\ldots,x_m;x]]}.
  \end{equation}
  Let  $v_p: F\xrightarrow{}\Z$ be a  discrete valuation of $F$ with $p$ as  an uniformizer.  We use  $\overline{F}_p$ and $F_p$ to denote  the residue field and the completion of $F$, respectively. Now we define the group $W_1(H_2^{m+1}(F_p))$ as the cokernel\[\text{coker} (\alpha': H_2^{m+1}(\overline{F}_p)\xrightarrow{} H_2^{m+1}(F_p)),\] where $\alpha'$ is  naturally induced by a Teichm\"uller lifting\footnote{We refer to  \cite[Appendix]{ex} for an explicit description of  Teichm\"uller lifting.} $\alpha: \overline{F}_p\xrightarrow{} F_p$.  Moreover, the group $W_1(H_2^{m+1}(F_p))$ is independent from the choice of  a Teichm\"uller lifting \cite[Corollary 2.6]{ex} .   The inclusion map $ F\xrightarrow{} F_p$ induces a homomorphism $ \beta_1:H_2^{m+1}(F)\xrightarrow{} H_2^{m+1}(F_p)$. Let   $\beta_2: H_2^{m+1}(F_p)\xrightarrow{} W_1(H_2^{m+1}(F_p))$ be  the homomorphism induced by projection.  Then, we have a (residue) map $\partial_p: H_2^{m+1}(F)\xrightarrow{} W_1(H_2^{m+1}(F_p))$ defined by $$\partial_p=\beta_2\circ \beta_1.$$ 

Now consider $F=\F(x)$ the rational function field in one variable $x$ over a field $\F$ of characteristic $2$.  Let $p\in \F[x]$ be a monic irreducible polynomial or $p=\frac{1}{x}$.  Let $F_p$ (\resp $\overline{F}_p$) denote the completion (\resp the residue field) of $F$ with respect to the $p$-adic valuation.  By  \cite[Theorem 4.2]{ex},  $H_2^m(\overline{F_p})$ is isomorphic to a quotient group of $W_1(H_2^{m+1}(F_p))$.  Moreover, the Scharlau's transfer $s_p': W_q(\overline{F_p})\xrightarrow{} W_q(\F)$, along with Kato's isomorphism (\ref{Ke}) induce  the homomorphism $s_p'': H_2^{m}(\overline{F_p})\xrightarrow{} H_2^m(\F)$. Consequently, we have the transfer map $s_p^*: W_1(H_2^{m+1}(F_p))\xrightarrow{} H_2^m(\F)$, which is the composition map $ W_1(H_2^{m+1}(F_p))\xrightarrow{} H_2^m(\overline{F_p})\xrightarrow{} H_2^m(\F)$ (we refer to Definition \ref{dt} for details).  Our main result is 
 the following theorem:
 \vspace{2mm}
 
\begin{thm}
    Let $\F$ be a field of characteristic $2$ and $F=\F(x)$ the rational function field in one variable $x$ over $\F$. Then, the following sequence in  exact:
    \[ 0\xrightarrow{} H_2^{m+1}(\F)\xrightarrow{\textcolor{black}{i}} H_2^{m+1}(F)\xrightarrow{ \bigoplus \partial_p} \bigoplus_{p} W_1(H_2^{m+1}(F_p))\xrightarrow{\bigoplus s^*_p} H_2^m(\F)\xrightarrow{}0, \]
     \textcolor{black}{ where i} is induced by the inclusion $\F\xrightarrow{} F$ and $p$ varies over $\frac{1}{x}$ and all monic irreducible polynomials  of $ \F[x]$. 
\end{thm}
\medskip

Now we  briefly summarize the proof of this theorem, \textcolor{black}{for which we will adapt some arguments from \cite{AJ}}. By \cite[Theorems 4.9, 4.10]{ex}, it  suffices to prove that \textcolor{black}{${\rm Im}(\oplus\partial_p)=  {\rm ker}(\oplus s^*_p)$}.  In Section \ref{S2}, we prove  a decomposition of $W_1(H_2^{m+1}(F_p))$ which will be needed in Section \ref{S3} to describe the transfer maps $s_p^*$. In Section \ref{S4}, 
we prove that $\oplus (s_q^*\circ \partial_q)$ vanishes on $ L_0+ L_0\wedge \overline{\frac{\dd p}{p}}$ for any fixed monic irreducible polynomial $p$. In fact, the map $ \oplus (s_q^*\circ  \partial_q)$ vanishes on $ H_2^{m+1}(\F(x))$ \textcolor{black}{(Theorem \ref{T51})},  which  proves that \textcolor{black}{${\rm Im}(\oplus \partial_q) \subset {\rm ker}(\oplus s^*_q)$}. Then, by some calculations we get  the other inclusion, which completes the proof of the theorem.
  
\section{A decomposition  of $W_1(H_2^{m+1}(F_p))$ }\label{S2}
  Let  $\F$ be a field of characteristic $2$.  A subset $\{t_i\hspace{1mm}|\hspace{1mm} i\in \I\}\subset \F$ is said to be a 
$2$-basis of $\F$ if $ \{\prod_{i\in \I}t_i^{\epsilon_i} \hspace{1mm} |\hspace{1mm} \epsilon_i\in\{0,1\}$ and $\epsilon_i=0$ for almost all $i\in \I\}$ is  an $\F^2$-basis of $\F$, where $\I$ is an ordered set.   

When $F=\F(x)$ and $p$ a monic  irreducible polynomial, the group $W_1(H_2^{m+1}(F_p))$ is given in  \cite[Theorem 4.2]{ex}.
 In this section we prove another  decomposition of $W_1(H_2^{m+1}(F_p))$, which will be needed in  Section \ref{S3} to introduce  the transfer maps $s_p^*$ that we talked about before.  

Let   $L_p$ be a   field of characteristic $2$, complete  with respect to a discrete valuation   $v_p:L_p\xrightarrow{} \Z$ and $p$ an uniformizer. Let $\overline{L}_p$ denote the residue field of $L_p$, and $ \overline{A}_p=\{s_i\hspace{1mm}|\hspace{1mm} i\in \I'\}$ a $2$-basis of $\overline{L}_p$, where $\I'$ is an ordered set.  Set  $ \I=\I'\cup \{\n\}$, where $s_{\n}$ denotes $p$. Now we extend the ordering of  $\I'$ to an ordering of $\I$ by taking $i<\infty$ for all $i\in \I'$.   Since $L_p$ is complete, we may consider $\overline{L}_p\subset L_p$ and  ${A}_p= \{s_i\hspace{1mm}|\hspace{1mm} i\in \I\}$ as  a $2$-basis of $L_p$.  
 
We define two groups  $ \overline{\mathcal{T}}_p= \bigoplus\limits_{|\I'|}(\Z/{2\Z})$ and ${\mathcal{T}}_p=\bigoplus\limits_{|\I|}(\Z/{2\Z})$.   Now $\mathcal{T}_p$ has \label{lp}lexicographically ordering with respect to $0<1$.  Namely, for any $I=(I_i)_{i\in \I}$ and $J=(J_i)_{i\in\I}\in  \mathcal{T}_p$ we have $I<J$ if there exists an element $k\in \I$ such that $I_k<J_k$ and $I_i=J_i$ for all $i>k$ {(we consider the ordering from right to left)}. Since $\overline{\mathcal{T}}_p$ can be considered  as a subgroup of  ${\mathcal{T}}_p$, the group $\overline{\mathcal{T}}_p$ is also lexicographically ordered. Let $L_p'=\{ l\in L_p \hspace{1mm}|\hspace{1mm} v_p(l)=0\}$ and for $I=(I_i)_{i\in \I}\in {{\mathcal{T}}_p}$,  let take $\tilde{I}= \{i\in \I\hspace{1mm}|\hspace{1mm}  I_i\neq 0\}$. For each  $m\in \N$, we define $ (\mathcal{T}_p)_m=\{ I\in {\mathcal{T}}_p\hspace{1mm}|\hspace{1mm} |\tilde{I}|=m \}$. Similarly, we define $(\overline{\mathcal{T}}_p)_m$.
For any positive integer $m$ and $I\in ({\mathcal{T}}_p)_m$, let $\langle\langle  s_{{I}}\rangle\rangle$   denote the $m$-fold bilinear Pfister form $ \bigotimes\limits_{i\in \tilde{I}} \langle\langle s_i\rangle\rangle_b$. For $J\in {\mathcal{T}}_p$, let $s^J$ denote the element  $\prod\limits_{i\in \tilde{J}}s_i$ of $L_p$. Now for each $m\in \N$, we have   \begin{equation*} ({\mathcal{T}}_p)_m\setminus (\overline{\mathcal{T}}_p)_m:=\{ I\in {\mathcal{T}}_p\hspace{1mm}|\hspace{1mm} |\tilde{I}|=m\hspace{1mm} \text{and}\hspace{1mm} I_{\n}=1\},\end{equation*} that is $p$ appears in $s^I$ for $I\in  (\mathcal{T}_p)_m\setminus (\overline{\mathcal{T}}_p)_m$. For any element $l\in L_p'$, let $\overline{l}$ denote the unique element of $\overline{L}_p$  such that $l\equiv \overline{l}\pmod{p}$. 
\begin{rem}\label{R21}
{Let us note that if $a\in L_p$ satisfies $v_p(a)>0$, then $a\in \wp(L_p):=\{x^2+x\mid x\in L_p\}$ (this is a direct application of the Hensel lemma).}
\end{rem}
Using these notations, our next result is the following proposition.

 \begin{prop}\label{P22}
      Let $p$, $L_p$, $L_p'$, $ \overline{L}_p$, $\mathcal{T}_p$, $\overline{\mathcal{T}}_p$ be as defined above. Suppose that the quadratic form  \[{Q}= \sum_{K\in ({\mathcal{T}}_p)_m\setminus (\overline{\mathcal{T}}_p)_m} \langle\langle s_{{K}}\rangle\rangle\otimes [1,a_{{K}}]+\sum_{L\in(\overline{\mathcal{T}}_p)_m}\langle\langle s_{{L}}\rangle\rangle\otimes[1,a_{{L}}] \in I_q^{m+2}(L_p),\] where
\begin{itemize} 
\item For $K\in ({\mathcal{T}}_p)_m\setminus (\overline{\mathcal{T}}_p)_m$, we have ${a_{{K}}= \sum\limits_{\{J\in \overline{\mathcal{T}}_p\hspace{1mm}\mid\hspace{1mm} J+K> K\}}\sum\limits_{r\geq 1}\frac{\overline{s^Jv_{r,K,J}^2}}{p^{2r}}}$ for 
$v_{r,K,J}\in L_p'$.
\item For $L\in (\overline{\mathcal{T}}_p)_m$, we have$$a_{{L}}=\sum\limits_{\{J\in \mathcal{T}_p\hspace{1mm}|\hspace{1mm} s^J=s^{J'}p\}} \sum\limits_{r\geq 0}\frac{\overline{ s^{J'} u_{r,L,J}^2}} {p^{2r+1}}+\sum\limits_{\{J\in \overline{\mathcal{T}}_p\mid J+L> L\}}\sum\limits_{r\geq 1}\frac{\overline{s^Jv_{r,L,J}^2}}{p^{2r}}$$for $ u_{r,L,J}$ and $v_{r,L,J}\in L_p'$.
\end{itemize}
Moreover,  $u_{r,I,J}=v_{r,I,J}=0$ for {almost} all $r,I,J$. Then, each $ a_{{I}}=0\in L_p$.  Consequently, $u_{r,I,J}=v_{r,I,J}=0$ for all $r,I,J$.
     
 \end{prop}
 \begin{proof} Let $R= \overline{L}_p[p^{-1}]$. Now $ \overline{L}_p\cong \alpha(\overline{L}_p)\subset L_p$, where $\alpha: \overline{L}_p\xrightarrow{} L_p$ is a\break Teichm\"uller lifting. Then $ R\cong \alpha(\overline{L}_p)[p^{-1}]\subset L_p$. Now we prove that  each $a_I$ is an  element  of the form $\sum\limits_{J\in \mathcal{T}_p,\,\textcolor{black}{J +I> I}} s^Jb^2_{I,J}\in R$, where  $b_{I,J}\in p^{-1}R$.   For any $I\in (\mathcal{T}_p)_m$, $a_I$ can be written as follows:
 \begin{equation*}
     a_I=\sum_{J\in \mathcal{T}_p\setminus \overline{\mathcal{T}}_p} A_{I,J}+ \sum_{J\in \overline{\mathcal{T}}_p,\, J+I>I} B_{I,J},
 \end{equation*} where $ A_{I,J}= \sum\limits _{r\geq 0} \frac{\overline{ s^{J'} u^2_{r,I,J}}}{p^{2r+1}}$ such that $ s^J= s^{J'}p$ and \textcolor{black}{ $ B_{I,J}= \sum\limits_{r\geq 1} \frac{\overline{s^J v^2_{r,I,J}}}{p^{2r}}$} for $ u_{r,I,J}$ and $v_{r,I,J}\in L_p'$.  \textcolor{black}{ If $J\in \overline{\mathcal{T}}_p$ and $x\in L_p'$, then we have  \begin{equation}\label{ep2}\overline{ s^J x^2}= s^J (\overline{x})^2\in \overline{L}_p.\end{equation}  Therefore,   $ B_{I,J}= \sum\limits_{r\geq 1} \frac{\overline{s^J v^2_{r,I,J}}}{p^{2r}}= \sum\limits_{r\geq 1} \frac{s^J (\overline{v_{r,I,J}})^2}{p^{2r}}= s^J b^2_{I,J}$, where $b_{I,J}= \sum\limits_{r\geq 1} \frac{\overline{v_{r,I,J}}}{p^r}\in p^{-1}R$. Thus, $B_{I,J}$ is in our required form.}   Now  $A_{I,J}$ can be expressed as follows:     \begin{align*}
      A_{I,J}&= \sum_{r\geq 0} \frac{\overline{s^{J'} u^2_{r,I,J}}}{p^{2r+1}}\quad (\text{where } s^J=  s^{J'}p)\\&= \sum_{r\geq 0} \frac{p(\overline{ s^{J'} u^2_{r,I,J}})}{p^{2r+2}} \\&\textcolor{black}{= \sum_{r\geq 0}\frac{  p s^{J'} (\overline{u_{r,I,J}})^2}{p^{2r+2}} \quad (\text{by (\ref{ep2})})}\\&= s^J b_{I,J}^2\in R,\quad (\text{where } b_{I,J}=\sum_{r\geq 0} \frac{\overline{u_{r,I,J}}}{p^{r+1}}\in p^{-1}R).
 \end{align*}  Note that  if $ I\in (\mathcal{T}_p)_m\setminus (\overline{\mathcal{T}}_p)_m$, then  $A_{I,J}=0$  for all \textcolor{black}{$ J\in \mathcal{T}_p\setminus \overline{\mathcal{T}}_p$}.  Therefore, $ A_{I,J}$ is nonzero only if $ I\in (\overline{\mathcal{T}}_p)_m$ and $J\in \mathcal{T}_p\setminus \overline{\mathcal{T}}_p$, and  in this case $J+I>I$. Hence, each $a_I$ belongs to $\sum\limits_{J\in \mathcal{T}_p,\, J+I>I} s^J (p^{-1} R)^2$.   Now \cite[Proposition 1.5]{AJ} implies that each $a_I=0\in L_p$. Thus, the proposition is proved. 
\end{proof}
\medskip
 
\textcolor{black}{ Let $\F$ be a field of characteristic $2$ with the  fixed $2$-basis  ${B}:= \{ t_i\hspace{1mm}|\hspace{1mm} i\in \I\}$ and $F=\F(x)$.   Let $p\in \F[x]$ be a monic irreducible polynomial of degree $d$ or $p=\frac{1}{x}$. Let us recall some notations from  \cite{ex}, which will be used repeatedly in the sequel. }
 \medskip 
  
 \textcolor{black}{ \noindent{\textbf{Notations:}} (1) For $d\geq 1$,  we define the sets $ \F[x]_{\leq d}:=\{ f(x)\in \F[x]\mid \deg f(x)\leq d\}$ and $ \F[x]_{<d}:=\{ f(x)\in \F[x]\mid \deg f(x)<d\}$.}
\medskip

\textcolor{black}{\noindent{(2)}  We denote by $F_p$ and $\overline{F}_p$ the completion  and the residue field of $F$ with respect to  the $p$-adic valuation, respectively. Recall that  we can see $\overline{F}_p$
  as a subfield of $F_p$.}
\medskip

\textcolor{black}{\noindent{(3)} Let $p$ be separable or $p=\frac{1}{x}$. Then, by \cite[Lemma 3.1]{AJ}, $B$ is also a $2$-basis of $\overline{F}_p$. Consequently, $B\cup\{p\}$ is a $2$-basis of $F_p$.}
\medskip

\textcolor{black}{\noindent{(4)} If $p$ is inseparable, then there exists some $i'\in \I$ such that  $t_{i'}\in \F(p)^2(B\setminus\{t_{i'}\})$. Then, $  (B\setminus  \{t_{i'}\})\cup\{\overline{x}\}$ is a $2$-basis of $\overline{F}_p$.  Therefore, $( B\setminus \{t_{i'}\})\cup \{p,x\}$ is a $2$-basis of $F_p$.}
\medskip
  
  \textcolor{black}{From now on, these fixed $2$-bases will be used for $\overline{F}_p$ and $F_p$ depending on the separability of $p$. }
 
{Let the group $T=\bigoplus_{|\I|} (\Z/2\Z)$ endowed with the lexicographic} ordering as defined in page \pageref{lp}.  For $I=(I_i)_{i\in \I}\in T$, we define  $\tilde{I}= \{ i\in \I \mid  I_i\neq 0\}$ and $ (T)_m=\{ I\in T\hspace{1mm} |\hspace{1mm} |\tilde{I}|=m\}$ {for any} $m\geq 1$. Moreover,   $ t^I$ denotes the element $ \prod_{i\in \tilde{I}}t_i\in \F$.   Hence $\{ t^I\hspace{1mm}|\hspace{1mm} I\in T\}$ is an $\F^2$-basis of $\F$.      Similarly,  we  define $\overline{T}_p=\oplus_{|\I_1|} (\Z/2\Z)$ and $T_p=\oplus_{|\I_2|} (\Z/2\Z)$ for two ordered sets $\I_1$ and $\I_2$.  If $p$ is separable, then $\I_1= \I$ and $\I_2=\I\cup\{\infty\}$, where $i<\infty$ for all $i\in \I$ and  $t_{\infty}=p$. If $p$ is inseparable, then $\I_1=\I\setminus \{i'\}\cup\{i_x\}$ and $ \I_2=\I_1\cup\{\infty\}$, where $i<i_x<\infty$ for all $i\in \I\setminus\{i'\}$ and $t_{i_x}=x$, $t_{\infty}=p$. \textcolor{black}{Here  $T_p$ and $\overline{T}_p$ are also lexicographically ordered}. For  $m\in \N$, we define $ (T_p)_m=\{I\in T_p\hspace{1mm}|\hspace{1mm} |\tilde{I}|=m\}$ and  similarly the set $(\overline{T}_p)_m$.  Note that $\{ t^I\hspace{1mm} |\hspace{1mm} I\in \overline{T}_p\}$ is an  $\F(p)^2$-basis of $\F(p)$ and $\{t^I\hspace{1mm} |\hspace{1mm}  I\in T_p\}$ is  an $F_p^2$-basis of $F_p$ for each $p$.  For $I\in ( T_p)_m$, ${\frac{\dd t_I}{t_I}}$ denotes the element $ {\frac{\dd t_{i_1}}{t_{i_1}}\wedge \ldots\wedge \frac{\dd t_{i_m}}{t_{i_m}}}$, where $\tilde{I}=\{i_1,\ldots,i_m\}$ with $i_1<i_2<\ldots<i_m$. Now $ T,\overline{T}_p$ can be considered as  subgroups of $T_p$, and for each $m\in \N$ we have: 

 \begin{equation*} (T_p)_m\setminus (\overline{T}_p)_m:=\{I\in T_p\mid |\tilde{I}|=m \;\text{and}\; I_{\infty}=1\}.\end{equation*}
 For any $g\in \F[x]$, {let $\overline{g}$ denote} the unique polynomial of degree {$\leq d-1$} such that $g\equiv \overline{g}\pmod{p}$, {where $d=\deg p$}. Using  these notations and Proposition \ref{P22}, we derive the following decomposition of $W_1(H_2^{m+1}(F_p))$. 
 \vspace{2mm}
 
\begin{thm}\label{T23}
Let $p\in \F[x]$ be a monic irreducible polynomial of degree $d\geq 1$. Then, any element $\phi\in W_1(H_2^{m+1}(F_p))$ can be expressed as follows:$$\phi= \sum\limits_{K\in ({T}_p)_m\setminus (\overline{T}_p)_m}\overline{ a_K\frac{\dd t_K}{t_K}}+ \sum\limits_{L\in (\overline{T}_p)_m} \overline{a_L \frac{\dd t_L}{t_L}}+ \phi'\wedge\overline{\frac{\dd p}{p}},$$where
\begin{itemize}
\item $\phi'\in H_2^{m}(\overline{F}_p)$,
\item For all $K\in ({T}_p)_m\setminus (\overline{T}_p)_m$, we have $a_{{K}}= \sum\limits_{\{J\in \overline{T}_p\mid J+K> K\}}\sum\limits_{r\geq 1}\frac{\overline{t^Jv_{r,K,J}^2}}{p^{2r}}$ for $ v_{r,K,J}\in \F[x]_{< d}$,
\item For $L\in (\overline{T}_p)_m$, we have$$a_{{L}}=\sum\limits_{\{J\in {T}_p\mid t^J=t^{J'}p\}} \sum\limits_{r\geq 0}\frac{\overline{ t^{J'} u_{r,L,J}^2}} {p^{2r+1}}+\sum\limits_{\{J\in \overline{T}_p \mid J+L> L\}}\sum\limits_{r\geq 1}\frac{\overline{t^Jv_{r,L,J}^2}}{p^{2r}},$$for $u_{r,L,J}, v_{r,L,J}\in \F[x]_{< d}$,  
\end{itemize}
such that $u_{r,I,J}=v_{r,I,J}=0$ for {almost} all $r,I,J$. Moreover, $\phi'$ and $u_{r,I,J}$, $v_{r,I,J}$ are unique for each $ r,I,J$.
\end{thm}

\begin{proof} \noindent{\bf 1. Existence of the decomposition.}
\medskip

Let $\phi\in W_1(H_2^{m+1}(F_p))$, and $\alpha: \overline{F}_p \xrightarrow{} F_p$ the Teichm\"uller lifting given as follows:
\begin{itemize}
\item for $p$ is separable: $\alpha( t_i)=t_i$ for all $i\in \I$. 
\item for $p$ is inseparable: $\alpha( t_i)=t_i$ for all $i\in \I\setminus\{i'\}$ and $\alpha(\overline{x})=x$.  
\end{itemize}

Since $\overline{F}_p\cong \alpha(\overline{F}_p)$, we consider $ \alpha(\overline{F}_p)\subset F_p$ as the residue field of $F_p$.  Therefore,  \cite[Theorem 2.3(1)]{ex} implies that  each element $\phi\in W_1(H_2^{m+1}(F_p))$ can be uniquely written as  ${\phi=\psi+\phi_2\wedge \overline{\frac{\dd p}{p}}}$, where $ \phi_2\in H_2^m(\alpha(\F(p)))$ and ${\psi= \sum\limits_{I\in (T_p)_m}\sum\limits_{\{J\in T_p\mid  J+I>I\}}\overline{t^J r^2_{I,J}\frac{\dd t_I}{t_I}}}$ for $r_{I,J}\in p^{-1}\alpha(\overline{F}_p)[p^{-1}] $. Hence, each $t^Jr^2_{I,J}$ can be written as ${t^Jr^2_{I,J}= \sum\limits_{i=1}^{l} \frac{\alpha( \overline{g}_i)}{p^i}}$ {for some $l\geq 1$}, where $g_i\in \F[x]$.  Since $\alpha(\overline{g}_i)=\sum\limits_{j\geq 0} g_{i,j}p^j$, with $g_{i,j}\in \F[x]_{<d}$,  we have  \[ \phi= \psi+\phi_2\wedge \overline{\frac{\dd p}{p}}\in W_1(H_2^{m+1}(F_p)),\] where    $\phi_2\in H_2^{m}(\alpha(\F(p)))$  and ${ \psi=\sum\limits_{I\in ({T}_p)_m}\sum\limits_{i\in \Z}\sum\limits_{l\leq i}\overline{\frac{ g_{i,l,I}}{p^l}\frac{\dd t_I}{t_I}}}$ with $ g_{i,l,I}\in \F[x]_{< d}$. We get the required  decomposition in two steps.\\

\noindent{\bf Step 1:} Here we prove that for any $f\in \F[x]_{< d}$, $I\in ({T}_p)_m$ and $l\in \Z$, we have in $H_2^{m+1}(F)$:\begin{equation}\label{fpl,equation}\overline{\frac{f}{p^l}\frac{\dd t_I}{t_I}}= \sum_{K\in ({T}_p)_m\setminus (\overline{T}_p)_m}\overline{ a_K\frac{\dd t_K}{t_K}}+ \sum_{L\in (\overline{T}_p)_m} \overline{ a_L \frac{\dd t_L}{t_L}}+ \sum_{M\in ({T}_p)_m}\overline{f_M\frac{\dd t_M}{t_M}}, \end{equation} where $f_M\in \F[x]$ and $a_K,a_L$ are as defined in Theorem \ref{T23}. If $l\leq 0$, then it is obvious. Now we {suppose that $l\geq 1$ and we proceed by induction on $l$.} For $ f\in \F[x]_{< d}$, using the $2$-bases of $\overline{F}_p$, we \textcolor{black}{find $ f_J,k\in \F[x]_{<d}$ for any $J\in \overline{T}_p$ such that} :\begin{equation}\label{de}
                f=\sum_{J\in \overline{T}_p} t^J f_J^2+pk.
            \end{equation} For $l=1,$ we have :\begin{align*}
                \overline{ \frac{f}{p}\frac{\dd t_I}{t_I}}&= \sum_{J\in \overline{T}_p} \overline{\frac{ t^J f_J^2}{p}\frac{\dd t_I}{t_I}}+\overline{k\frac{\dd t_I}{t_I}}\\&=\sum_{\{J'\in {T}_p\mid J'_{\n}=1\}} \overline{ \frac{t^{J'} f_J^2}{p^2}\frac{\dd t_I}{t_I}}+\overline{k\frac{\dd t_I}{t_I}}\quad \textcolor{black}{(\text{{where }}  t^{J'}=t^J)}.\end{align*} 

Since $J'\neq 0$, it follows from \cite[Remark 2.4]{ex} that we can replace $I$ by $I'$ such that $ J'+I'> I'$, keeping the separable part same. Since $J'_{\n}=1$ and $J'+I'>I'$, it follows that $I'_{\n}=0$. Therefore, we have: \begin{align*}
\overline{ \frac{f}{p}\frac{\dd t_I}{t_I}}&=\sum_{I'\in (\overline{T}_p)_m}\sum_{\{ J'\in {T}_p\mid J'_{\n}=1,\, J'+I'> I'\}} \overline{ \frac{t^{J'}f^2_{J}}{p^2}\frac{\dd t_{I'}}{t_{I'}}}+\overline{k\frac{\dd t_I}{t_I}}\\& = \sum_{I'\in (\overline{T}_p)_m}\sum_{{J\in \overline{T}_p}} \overline{\frac{t^Jf_J^2}{p}\frac{\dd t_{I'}}{t_{I'}}}+\overline{k\frac{\dd t_I}{t_I}}\\& = \sum_{I'\in (\overline{T}_p)_m}\sum_{J\in \overline{T}_p}\overline{ \frac{\overline{ t^Jf^2_J}}{p}\frac{\dd t_{I'}}{\dd t_{I'}}}+ \sum_{I'\in (\overline{T}_p)_m}\sum_{J\in \overline{T}_p}\overline{{h_{J}} \frac{\dd t_{I'}}{t_{I'}}}+\overline{k\frac{\dd t_I}{t_I}},
                \end{align*} where $ h_{J}\in \F[x]$.
                 Now the sum in the right side is in  our required form as in (\ref{fpl,equation}). 
               
\noindent{Suppose} that for some $r\geq 1$, we may write  $ {\overline{ \frac{f}{p^l}\frac{\dd t_I}{t_I}}} $ as the sum in (\ref{fpl,equation}) for any $l\leq r$ and $f\in \F[x]_{< d}$. We have to prove the same for $ {\overline{ \frac{f}{p^{r+1}}\frac{\dd t_I}{t_I}}} $. We treat two cases  depending on {$r$ is even or odd.}
\medskip

\noindent{(i)} {Suppose that $r$ is odd and put $r+1=2e$.} Then, using the value of $f$ from (\ref{de}), we get: 
\begin{align*}
\overline{\frac{f}{p^{r+1}}\frac{\dd t_I}{t_I}}&= \sum_{J\in \overline{T}_p} \overline{ \frac{t^Jf_J^2}{p^{2e}}\frac{\dd t_I}{t_I}}+\overline{\frac{k}{p^r}\frac{\dd t_I}{t_I}} =\sum_{0\neq J\in \overline{T}_p} \overline{ \frac{t^Jf_J^2}{p^{2e}}\frac{\dd t_I}{t_I}}+\overline{ \frac{f_0^2}{p^{2e}}\frac{\dd t_I}{t_I}}+\overline{\frac{k}{p^r}\frac{\dd t_I}{t_I}}. \end{align*} { Since in the first sum $J\neq 0$, {it follows from \cite[Remark 2.4]{ex} that} we can change $I$ to $I'$ such that $J+I'> I'$ keeping the separable part same. Therefore, we get} \begin{align*} \overline{\frac{f}{p^{r+1}}\frac{\dd t_I}{t_I}}&= \sum_{I'\in ({T}_p)_m}\sum_{\{J\in \overline{T}_p\mid J+I'> I'\}} \overline{ \frac{ t^J f^2_J}{p^{2e}}\frac{\dd t_{I'}}{t_{I'}}} + \overline{ \frac{f_0}{p^e}\frac{\dd t_I}{t_I}}+\overline{\frac{k}{p^r}\frac{\dd t_I}{t_I}}.\end{align*}{Since for any $J\in \overline{T}_p$, $\deg( t^Jf_J^2)\leq 2d-1$, we can write $ t^Jf_J^2= ph_{1,J}+h_{2,J}$ with $\deg h_{i,J}\leq d-1$. Thus, we obtain:}\begin{align*} \overline{\frac{f}{p^{r+1}}\frac{\dd t_I}{t_I}}&= \sum_{I'\in ({T}_p)_m}\sum_{\{J\in \overline{T}_p\mid J+I'> I'\}} \overline{ \left(\frac{ h_{2,J}}{p^{2e}}+\frac{h_{1,J}}{p^{r}}\right)\frac{\dd t_{I'}}{t_{I'}}} + \overline{ \frac{f_0}{p^e}\frac{\dd t_I}{t_I}}+\overline{\frac{k}{p^r}\frac{\dd t_I}{t_I}}\\&= \sum_{I'\in ({T}_p)_m, J\in \overline{T}_p\atop J+I'> I'} \overline{ \frac{ \overline{ t^J f_J^2}}{p^{2e}}\frac{\dd t_{I'}}{t_{I'}}} + \sum_{I'\in ({T}_p)_m, J\in \overline{T}_p\atop J+I'> I'} \overline{ \frac{ h_{1,J}}{p^{r}}\frac{\dd t_{I'}}{t_{I'}}} +\overline{ \frac{f_0}{p^e}\frac{\dd t_I}{t_I}}+\overline{\frac{k}{p^r}\frac{\dd t_I}{t_I}}\\&= \sum_{K\in ({T}_p)_m\setminus (\overline{T}_p)_m, J\in \overline{T}_p \atop J+K> K}\overline{\frac{\overline{t^J f^2_J}}{p^{2e}} \frac{\dd t_{K}}{t_{K}}}+ \sum_{L\in (\overline{T}_p)_m, J\in \overline{T}_p\atop J+L> L}\overline{\frac{\overline{t^J f^2_J}}{p^{2e}} \frac{\dd t_{L}}{t_{L}}}+B,
\end{align*}where ${B= \sum\limits_{I'\in ({T}_p)_m, J\in \overline{T}_p\atop J+I'> I'} \overline{ \frac{ h_{1,J}}{p^{r}}\frac{\dd t_{I'}}{t_{I'}}} +\overline{ \frac{f_0}{p^e}\frac{\dd t_I}{t_I}}+\overline{\frac{k}{p^r}\frac{\dd t_I}{t_I}}}$. Now {the first and the second sums} of the right  side are in required form, and $B $ is in required form by induction hypothesis. So we are done for {$r$ odd}.\\

\noindent{(ii)} {Suppose that $r$ is even. Then, $r+1$ is odd and this case will be similar to  the case $l=1$. }We will also have to use induction hypothesis. 
\medskip
            
Finally, we can say that for any $f\in \F[x]_{< d}, I\in ({T}_p)_m$ and $l\in \Z$,  the equation (\ref{fpl,equation}) holds in $H_2^{m+1}(F)$.  Now using Remark \ref{R21} in $W_1 (H_2^{m+1}(F_p))$, we get$$\phi= \sum_{K\in ({T}_p)_m\setminus (\overline{T}_p)_m}\overline{ a_K\frac{\dd t_K}{t_K}}+ \sum_{L\in (\overline{T}_p)_m} \overline{ a_L \frac{\dd t_L}{t_L}}+ \sum_{M\in ({T}_p)_m}\overline{f_M\frac{\dd t_M}{t_M}}+\phi_2\wedge\overline{\frac{\dd p}{p}},$$  where $ f_M\in \F[x]_{<d}, \phi_2\in H_2^m(\alpha(\F(p)))$ and $a_K, a_L$ are as defined in the statement of {the theorem}.\\
            
\noindent{\bf Step 2:} Let ${ \phi_2= \sum\limits_{I\in (\overline{T}_p)_{m-1}}\overline{ h_I\frac{\dd t_I}{t_I} }},$ where $ h_I\in \alpha( \F(p))$. Therefore, $ h_I= h'_I+h''_I$, where $h'_I\in \F[x]_{<d}$ and $v_p(h''_I)>0$.  Hence, we have    ${\phi_2=  \sum\limits_{I\in (\overline{T}_p)_{m-1}}  \overline{h_I'\frac{\dd t_I}{t_I}}}\in H_2^{m+1}(F_p)$. Moreover, $ {\overline{ f_M\frac{\dd t_M}{t_M}}}\in  H_2^{m+1}(\F(p)) $ or ${H_2^{m}(\F(p))\wedge \overline{\frac{\dd p}{p}}}$ according as $M_{\n}=0$ or $1$.
            Therefore, we  can express 
             $\phi\in W_1(H_2^{m+1}(F_p))$ as follows:
\begin{equation}\label{phi} \phi=\sum_{K\in ({T}_p)_m\setminus (\overline{T}_p)_m}\overline{ a_K\frac{\dd t_K}{t_K}}+ \sum_{L\in (\overline{T}_p)_m} \overline{ a_L \frac{\dd t_L}{t_L}}+ \phi'\wedge\overline{\frac{\dd p}{p}},\end{equation} where $\phi'\in H_2^m(\F(p))$.
\newpage

\noindent{\bf 2. Uniqueness of the decomposition:}  
\medskip

Let $ \phi=\sum\limits_{K\in ({T}_p)_m\setminus (\overline{T}_p)_m}\overline{ a_K\frac{\dd t_K}{t_K}}+ \sum\limits_{L\in (\overline{T}_p)_m} \overline{ a_L \frac{\dd t_L}{t_L}}+ \phi'\wedge\overline{\frac{\dd p}{p}}= 0\in W_1(H_2^{m+1}(F_p))$.  Then, in $H_2^{m+1}(F_p)$, {we have $\phi= \phi_1$ for some $\phi_1\in H_2^{m+1}(\alpha(\overline{F}_p))$.} Therefore, we get
\begin{equation}  
\sum_{K\in ({T}_p)_m\setminus (\overline{T}_p)_m}\overline{ a_K\frac{\dd t_K}{t_K}}+ \sum_{L\in (\overline{T}_p)_m} \overline{ a_L \frac{\dd t_L}{t_L}}+ \phi'\wedge\overline{\frac{\dd p}{p}}+\phi_1=0\in H_2^{m+1}(F_p).
\label{uniq}
\end{equation}

Now we take a suitable finite separable extension $L_p$ of $F_p$ such that $(\phi_1)_{L_p}=0$ and $(\phi')_{L_p}=0$. The $2$-basis of $F_p$ remains a $2$-basis of $L_p.$ The 
$p$-adic valuation of $F_p$ can be {uniquely extended to $L_p$, which} will be complete with respect to this valuation.  \textcolor{black}{Let $\overline{L}_p$ be the residue field  of $L_p$. Then, $\overline{L}_p$ is a separable extension of $\overline{F}_p$, and thus the $2$-basis of $\overline{F}_p$ remains a $2$-basis of $ \overline{L}_p$.} By (\ref{uniq}), we have $${{(\phi)_L}_p=\sum\limits_{K\in ({T}_p)_m\setminus (\overline{T}_p)_m}\overline{ a_K\frac{\dd t_K}{t_K}}+ \sum\limits_{L\in (\overline{T}_p)_m} \overline{ a_L \frac{\dd t_L}{t_L}}=0\in H_2^{m+1}(L_p)}.$$ 

Using the Kato's isomorphism (\ref{Ke}), it follows from Proposition \ref{P22} that $a_K=a_L=0\in L_p$, and thus $a_K=a_L=0\in F_p$. 

Recall that ${a_{{K}}= \sum\limits_{\{J\in \overline{T}_p \mid J+K> K\}}\sum\limits_{r\geq 1}\frac{\overline{t^Jv_{r,K,J}^2}}{p^{2r}}}$, where $ v_{r,K,J}\in \F[x]_{< d}$ for all $K\in ({T}_p)_m\setminus (\overline{T}_p)_m$.  Therefore, for each $r\geq 1$, we have $$\sum\limits_{\{J\in \overline{T}_p\hspace{1mm}|\hspace{1mm} J+K> K\}}\overline{ t^J v^2_{r,K,J}}=0\in \overline{F}_p.$$

Since  $ \{t^J \mid J\in \overline{T}_p\}$ is an $\overline{F}^2_p$-basis of $\overline{F}_p$, and $\deg v_{r,K,J}\leq d-1$, it follows that each $v_{r,K,J}=0$. 

Again ${ a_{{L}}=\sum\limits_{\{J\in {T}_p\mid t^J=t^{J'}p\}} \textcolor{black}{\sum\limits_{r\geq 0}\frac{\overline{ t^{J'} u_{r,L,J}^2}} {p^{2r+1}}}+\sum\limits_{\{J\in \overline{T}_p \mid J+L> L\}}\sum\limits_{r\geq 1}\frac{\overline{t^Jv_{r,L,J}^2}}{p^{2r}}}$,  where   $ u_{r,L,J}, v_{r,L,J}\in \F[x]_{< d}$ and $L\in (\overline{T}_p)_m$. Here also we get $ u_{r,L,J}=v_{r,L,J}=0$. Finally, we have  ${\phi'\wedge \overline{\frac{\dd p}{p} }+\phi_1=0\in H_2^{m+1}(F_p)} $. Applying the  residue map $\zeta$ of \cite[Definition 6.12]{ex}, we get $\zeta\left( \phi_1+ \phi'\wedge\overline{\frac{\dd p}{p}}\right)=\phi'=0\in H_2^m(\F(p))$. Thus, the uniqueness holds.
\end{proof}

\begin{rem} {Theorem  \ref{T23} and \cite[Theorem 4.2]{ex} give two different decompositions of $W_1(H_2^{m+1}(F_p))$.  From now on, we will use both decompositions as needed.}
\end{rem}
      
\section{Introducing the  transfer maps $s_p^{*}$}\label{S3}
{Here} also we  use the same $2$-bases for $ \F$, $\overline{F}_p$ and $F_p$, where  $F=\F(x)$ and  $p$  is a monic irreducible polynomial or  {$p=\frac{1}{x}$}, as defined in Section \ref{S2}. We also use  the notations $T,\overline{T}_p$ and ${T}_p$ as before.  Note that if $I\in {T}_p$, then $I_{\infty}$ denotes the $p$-th component of $I$. If  $p$ is {inseparable} and $I\in {T}_p$, then $I_x$ denotes the $x$-th component of $I$. Therefore, $I_x=1$ implies that $ t^I$ is a multiple of $x.$ Now we recall some notations from \cite{ex}.
\medskip
 
 \noindent{\textbf{Notations:}}
 \noindent{(1)} For $d\geq 1$, $L_d$ is the subgroup of $H_2^{m+1}(F)$ generated by the elements  $ {\overline{ \frac{h}{u^e} \frac{\dd f_1}{f_1} \wedge \ldots\wedge \frac{\dd f_m}{f_m}}}$, where $h\in \F[x]$, $ e\geq 0$ and {$0\neq f_i\in \F[x]_{\leq d}$.}
\medskip
 
\noindent{(2)} $L_0$ is the subgroup  of $H_2^{m+1}(F)$ generated by  the elements $ {\overline{ h\frac{\dd c_1}{c_1}\wedge \ldots\wedge\frac{\dd c_{m-1}}{c_{m-1}} \wedge \frac{\dd f_m}{f_m}}}$, where $c_i\in \F^*$, $f_m\in \F^*\cup\{x\}$ and $ h\in \F[x]$ ({\resp} $h\in x\F[x] $) if $f_m\in \F^* $ ({\resp}  if $f_m=x$).
\medskip
 
\noindent{(3)} For {$p\in \F[x]$} monic irreducible, $S_p$ denotes the subgroup of $ H_2^{m+1}(F)$ generated by the elements $ {\overline{\frac{h}{p^e} \frac{\dd t_I}{t_I}}}$, where $h\in \F[x]$, $e\geq 0$ and $I\in (T)_m$. 
\medskip
 
\noindent{(4)}  Similarly, $S_{\frac{1}{x}}$ denotes the subgroup of $ H_2^{m+1}(F)$ generated by  the elements $ {\overline{{xh}\frac{\dd t_I}{t_I}}}$, where $h\in \F[x]$ and $I\in (T)_m$.
\medskip
 
\noindent{(5)} If $p$ is inseparable, then $\tilde{S}_p$ denotes the subgroup of $ H_2^{m+1}(F)$ generated by the elements  $ {\overline{\frac{h}{p^e} \frac{\dd t_I}{t_I}}}$, where $h\in \F[x]$, $e\geq 0$ and $I\in (\overline{T}_p)_m$.\qed

Whenever we use $S_p\wedge \overline{\frac{\dd p}{p}}$ or $\tilde{S}_p\wedge \overline{\frac{\dd p}{p}}$, then {$S_p$ and $\tilde{S}_p$ are considered as subgroups of $H_2^{m}(F)$} to balance the total wedge.  Moreover,  for any $g\in \F[x]$, let $\overline{g}$ denotes the unique polynomial of degree $<\deg p$ such that $g\equiv \overline{g}\pmod{p}$. Using these notations, we give some technical definitions. 
\medskip

\begin{defn}\label{Up} {\bf (The separable case).} Suppose that $p$ is separable.

\noindent{(i)} For $r$ odd, let ${S'_{p,r}\subseteq S_p+ S_p\wedge \overline{\frac{\dd p}{p}}}$ be the subgroup  generated by   the  forms ${\overline{ \frac{\overline{t^Js^2_{r,I,J}}}{p^r} \frac{\dd t_I}{t_I}}},$ where  $ I\in (\overline{T}_p)_m$, $J\in \overline{T}_p$ and $  s_{r,I,J}\in \F[x]_{<  \deg p}$ .\\

\noindent{(ii)} For $r$ even, let ${S'_{p,r}\subseteq S_p+S_p\wedge \overline{\frac{\dd p}{p}}}$ be the subgroup generated by the  forms $ { \overline{\frac{\overline{t^Js^2_{r,I,J}}}{p^r}\frac{\dd t_I}{t_I}}}$ and ${ \overline{\frac{\overline{t^{J'} s^2_{r,I',J'}}}{p^r} \frac{\dd t_{I'}}{t_{I'}}\wedge \frac{\dd p}{p}}}$,  where  $ I\in ({T})_m,I'\in ({T})_{m-1}$  and $ J,J'\in {T}$ such that $ J+I> I$, $J'+I'> I'$. Moreover, $ s_{r,I,J}$, $s_{r,I',J'}\in \F[x]_{< \deg p}$.

Recall that in this case $ \overline{T}_p=T$.
\end{defn}

\begin{defn} {\bf (The inseparable case).} Suppose that $p$ is inseparable.
\medskip

\noindent{(i)} For $r$ odd, we define:
\begin{itemize}
\item ${S'_{p,r}\subseteq  \tilde{S}_p+\tilde{S}_p\wedge \overline{\frac{\dd p}{p}}}$ as the subgroup generated by the forms ${\overline{ \frac{\overline{ t^Js^2_{r,I,J}}}{p^r} \frac{\dd t_I}{t_I}}},$ where  $ I\in (\overline{T}_p)_m, J\in \overline{T}_p$ and  $ s_{r,I,J}\in \F[x]_{<  \deg p}$.

\item $ S_{p,r}^0\subseteq S'_{p,r}$ as the subgroup generated by the forms   ${\overline{ \frac{\overline{t^Js^2_{r,I,J}}}{p^r} \frac{\dd t_I}{t_I}}},$ where  $ J\in \overline{T}_p$, $s_{r,I,J}\in \F[x]_{<\deg p}$ and $ I\in (\overline{T}_p)_m, $ with $I_x=0$.
\end{itemize}

\noindent{(ii)} For $r\geq 2$ even, we define:
\begin{itemize}
\item ${S'_{p,r}\subseteq \tilde{S}_p+\tilde{S}_p\wedge \overline{\frac{\dd p}{p}}}$ as the subgroup generated by the forms ${\overline{\frac{\overline{t^Js^2_{r,I,J}}}{p^r}\frac{\dd t_I}{t_I}}}$ and ${\overline{\frac{\overline{t^{J'}s^2_{r,I',J'}}}{p^r}\frac{\dd t_{I'}}{t_{I'}}\wedge \frac{\dd p}{p}}}$, where  $I\in (\overline{T}_p)_m$, $I'\in (\overline{T}_p)_{m-1}$ and   $J,J'\in \overline{T}_p$ such that  $ J+I> I, J'+I'> I'$. Moreover, $ s_{r,I,J},s_{r,I',J'}\in \F[x]_{<\deg p}$.

\item $ S_{p,r}^0\subseteq S'_{p,r}$ as the  subgroup 
generated by the elements $ {\overline{\frac{\overline{t^J s^2_{r,I,J}}}{p^r}\frac{\dd t_I}{t_I}}}$ and $ {\overline{\frac{\overline{t^{J'}s^2_{r,I',J'}}}{p^r}\frac{\dd t_{I'}}{t_{I'}}\wedge \frac{\dd p}{p}}}$, where $J\in \overline{T}_p, J'\in \overline{T}_p$ and $I\in (\overline{T}_p)_m, I'\in (\overline{T}_p)_{m-1}$, with $I_x=I'_x=0$ such that $ J+I> I, J'+I'> I'$. Moreover, $ s_{r,I,J}$, $s_{r,I',J'}\in \F[x]_{< \deg p}$.
\end{itemize}
\label{Up2}
\end{defn}
\begin{defn} {\bf (The case $p=1/x$).}\\
\noindent{(i)} For $r$ odd, let ${ S'_{\frac{1}{x},r}\subseteq S_{\frac{1}{x}}+ S_{\frac{1}{x}}\wedge \overline{\frac{\dd x}{x}}}$ be the subgroup generated by  the forms $ {\overline{ t^Jc_{r,I,J}^2 x^r \frac{\dd t_I}{t_I}}},$ where $I\in ({T})_m$, $J\in T$ and $c_{r,I,J}\in \F$.
\medskip

\noindent{(ii)} For $r$ even, let ${S'_{\frac{1}{x},r}\subseteq S_{\frac{1}{x}}+ S_{\frac{1}{x}}\wedge \overline{\frac{\dd x}{x}} }$ be the subgroup generated by the  following type of  forms ${\overline{ t^{J}c_{r,I,J}^2 x^r \frac{\dd t_I}{t_I}}} $ and $ {\overline{ t^{J'}c_{r,I',J'}^2 x^r \frac{\dd t_{I'}}{t_{I'}}\wedge \frac{\dd x}{x}}}$, where  $ I\in(T)_m,I'\in (T)_{m-1} $ and $J',J\in T$, such that $ J+I>I$ and $J'+I'>I'$. Moreover,   $ c_{r,I,J},c_{r,I',J'}\in \F$.
\medskip
         
\noindent{(iii)} Finally,  we define the subgroup  $ {U_{p}:= \sum\limits_{r\geq 1} S'_{p,r}}\subseteq H_2^{m+1}(F_p)$ for all $p$  monic irreducible and ${p=\frac{1}{x}}.$ If $p$ is inseparable,  then we define  ${U_{p}^0:= \sum\limits_{r \geq 1} S_{p,r}^0}$.
\end{defn}

In the following lemma we establish the  relations between  $U_p$ and   $ {S_p\wedge \overline{\frac{\dd x}{x}}}$, where $p$ is  monic irreducible.
 
 \begin{lemma}\label{L32} {Let \textcolor{black}{$p\in \F[x]$} be an irreducible monic polynomial or $p=\frac{1}{x}$.}\\
 
\noindent{(i)} If $p$ is separable, then $S_p+S_p\wedge \overline{\frac{\dd p}{p}}\subseteq U_p+ L_0+L_0\wedge \overline{\frac{\dd p}{p}}.$ \\
 
\noindent{(ii)} If $p$ is separable, then $ S_p+ S_p\wedge \overline{\frac{\dd x}{x}}\subseteq U_p+L_0+L_0\wedge \overline{\frac{\dd x}{x}}+L_0\wedge \overline{\frac{\dd p}{p}}.$\\

\noindent{(iii)} If $p$ is inseparable, then $ \tilde{S}_p+ \tilde{S}_p\wedge \overline{\frac{\dd p}{p}}\subseteq U_p+L_0+L_0\wedge \overline{\frac{\dd x}{x}}+ L_0\wedge \overline{ \frac{\dd p}{ p}}+L_0\wedge \overline{\frac{\dd x}{x}\wedge \frac{\dd p}{p}}.$\\

\noindent{(iv)} If $p$ is inseparable, then $ S_p+S_p\wedge \overline{\frac{\dd x}{x}}\subseteq  U_p+L_0+L_0\wedge \overline{\frac{\dd x}{x}}+ L_0\wedge \overline{ \frac{\dd p}{ p}}+L_0\wedge \overline{\frac{\dd x}{x}\wedge \frac{\dd p}{p}}.$\\

\noindent{(v)}  The groups  $U_p=\bigoplus\limits_{r\geq 1} S'_{p,r}$ and $U_p^0=\bigoplus\limits_{r\geq 1}S_{p,r}^0$ are direct sums.\\

\noindent{(vi)}  Let $\overline{U_p}$ and $\overline{U_0}$ be the image of $U_p$ and $U_0$ in $W_1(H_2^{m+1}(F_p))$, respectively. Then, for each $p$ monic irreducible or ${\frac{1}{x}}$, we have: $$W_1(H_2^{m+1}(F_p))/{\overline{U_p}}\cong H_2^m(\F(p))\wedge \overline{\frac{\dd p}{p}},$$ where ${ H_2^m(\F(p))\wedge \overline{\frac{\dd p}{p}}}$ is considered as a subgroup of $ H_2^{m+1}(F_p).$
\end{lemma}
 \begin{proof}
 (i) {Suppose that $p$ is separable of degree $d$.} The generators  of ${S_p+ S_p\wedge \overline{\frac{\dd p}{p}}}$ are  $ {\overline{\frac{h_l}{p^l}\frac{\dd t_I}{t_I}}}$, where  $I\in ({T}_p)_m$ and $l\geq 0$. Moreover, \textcolor{black}{we may suppose $h_l\in \F[x]_{< d}$ for all $l> 0$ and $h_0\in \F[x]$.}\\
{Now}  $ {\overline{h_0\frac{\dd t_I}{t_I}}}\in L_0$ or ${L_0\wedge \overline{\frac{\dd p}{p}}}$ according as $I_{\n}=0$ or $1$. For any $l\geq 1$ and $ h_l\in \F[x]_{< d}$, {we deduce from} (\ref{fpl,equation}) in $H_2^{m+1}(F)$:
              \begin{equation*}
                  \overline{ \frac{h_l}{p^l}\frac{\dd t_I}{t_I}}= \sum_{K\in ({T}_p)_m\setminus (\overline{T}_p)_m }\overline{a_K\frac{\dd t_K}{t_K}}+\sum_{L\in (\overline{T}_p)_m}\overline{a_L\frac{\dd t_L}{t_L}}+\sum_{M\in ({T}_p)_m}\overline{f_M\frac{\dd t_M}{t_M}},
              \end{equation*}
              where ${f_M}\in \F[x]$ and 
\[\begin{cases}a_K=\sum\limits_{\{J\in \overline{T}_p\mid J+K> K\}} \sum\limits_{r\geq 1}\frac{\overline{t^Jv^2_{r,K,J}}}{p^{2r}}\\a_L= \sum\limits_{\{ J\in {T}_p\mid t^J=t^{J'}p\}}\textcolor{black}{\sum\limits_{r\geq 0}\frac{\overline{t^{J'}u^2_{r,L,J}}}{p^{2r+1}}}+ \sum\limits_{\{ J\in \overline{T}_p\mid J+L> L\}}\sum\limits_{r\geq 1}\frac{\overline{ t^J v^2_{r,L,J}}}{p^{2r}}\end{cases}\]
              for $u_{r,I,J}$, $v_{r,I,J}\in \F[x]_{< d}$.  Now $\sum\limits_{M\in ({T}_p)_m}\overline{f_M\frac{\dd t_M}{t_M}}\in L_0+L_0\wedge \overline{\frac{\dd p}{p}}$ by the case $l=0$. \\
Clearly, by   Definition \ref{Up}(i), we have   $ {\overline{a_L\frac{\dd t_L}{t_L}}}\in U_p$ for  each  $L\in (\overline{T}_p)_m$. For any $K\in ({T}_p)_m\setminus(\overline{T}_p)_m$, we get  :\begin{align*}
    \overline{ a_K\frac{\dd t_K}{t_K}}&= \sum_{\{J\in \overline{T}_p\hspace{1mm}|\hspace{1mm} J+K> K\}} \sum_{r\geq 1} \overline{\frac{ \overline{t^J v^2_{r,K,J}}}{p^{2r}}\frac{\dd t_K}{t_K}}.\end{align*} Since $K_{\n}=1$, $ t^K =t^{K'}p$, for some $K'\in (\overline{T}_p)_{m-1}$. Therefore, by definition $ {\frac{\dd t_K}{t_K}=\frac{\dd t_{K'}}{t_{K'}}\wedge \frac{\dd p}{p}}$. Since $J\in \overline{T}_p$ and   $J+K>K$, it follows   that $J+K'>K'$. Hence, we get  \begin{align*}\overline{ a_K\frac{\dd t_K}{t_K}} &=\sum_{\{J\in \overline{T}_p \mid J+K'> K'\}} \sum_{r\geq 1} \overline{\frac{ \overline{t^J v^2_{r,K,J}}}{p^{2r}}\frac{\dd t_{K'}}{t_{K'}}\wedge \frac{\dd p}{p}}\in U_p.
\end{align*} 
Thus, for any $l\geq 1$, $h_l\in \F[x]_{< d}$ and $I\in ({T}_p)_m$,  we have: $ {\overline{\frac{h_l}{p^l}\frac{\dd t_I}{t_I}}}\in U_p+L_0+L_0\wedge \overline{\frac{\dd p}{p}}.$
         Therefore,   $ {S_p+S_p\wedge \overline{\frac{\dd p}{p}}\subseteq U_p+L_0+L_0\wedge \overline{\frac{\dd p}{p}}}$.

(ii) Suppose that $p$ is separable of degree $d$. By \cite[Proposition 3.8(2)]{ex},  we have the equality$${S_p+S_p\wedge\overline{\frac{\dd x}{x}}+L_0\wedge \overline{\frac{\dd x}{x}}= S_p+S_p\wedge \overline{\frac{\dd p}{p}}+L_0\wedge \overline{\frac{\dd x}{x}}}.$$By the previous part, $ S_p+S_p\wedge \overline{\frac{\dd p}{p}}\subseteq U_p+L_0+L_0\wedge \overline{\frac{\dd p}{p}}.$
Combining these two relations we have :\[S_p+S_p\wedge\overline{\frac{\dd x}{x}}\subseteq U_p+ L_0+ L_0\wedge \overline{\frac{\dd x}{x}}+L_0\wedge \overline{\frac{\dd p}{p}}.\]

(iii) Suppose that $p$ is inseparable of degree $d$. The generators of $\tilde{S}_p+\tilde{S}_p\wedge \overline{\frac{\dd p}{p}}$ are $ \overline{ \frac{h_l}{p^l}\frac{\dd t_I}{t_I}}$, where  $I\in ({T}_p)_m$ and $l\geq 0$. Moreover, $h_l\in \F[x]_{< d}$ for $l\geq 1$ and $h_0\in \F[x]$.
 Now for $l=0$, we have:
         \begin{equation*}
             \overline{ h_0\frac{\dd t_I}{t_I}}\in \begin{cases}
             L_0\wedge \overline{\frac{\dd x}{x}\wedge \frac{\dd p}{p}}& \text{if}\; I_x=1,I_{\n}=1,\\ L_0\wedge \overline{\frac{\dd p}{p}} & \text{if}\; I_x=0,I_{\n}=1,\\L_0\wedge \overline{\frac{\dd x}{x}} & \text{if}\; I_x=1, I_{\n}=0,\\
                 L_0 & \text{if}\; I_x=0,I_{\n}=0.
                 \end{cases}
         \end{equation*}
 By  (\ref{fpl,equation}), for any $l\geq 1$ and $h_l\in \F[x]_{< d}$, we have in $H_2^{m+1}(F)$: \begin{equation*}
                  \overline{ \frac{h_l}{p^l}\frac{\dd t_I}{t_I}}= \sum_{K\in ({T}_p)_m\setminus (\overline{T}_p)_m}\overline{a_K\frac{\dd t_K}{t_K}}+\sum_{L\in (\overline{T}_p)_m}\overline{a_L\frac{\dd t_L}{t_L}}+\sum_{M\in ({T}_p)_m}\overline{f_M\frac{\dd t_M}{t_M}},
              \end{equation*}
              where  $f_M\in \F[x]$ and
\[\begin{cases}a_K=\sum\limits_{\{J\in \overline{T}_p\mid J+K> K\}} \sum\limits_{r\geq 1}\frac{\overline{t^Jv^2_{r,K,J}}}{p^{2r}},\\a_L= \sum\limits_{\{ J\in {T}_p\mid t^J=t^{J'}p\}}\textcolor{black}{\sum\limits_{r\geq 0}\frac{\overline{t^{J'}u^2_{r,L,J}}}{p^{2r+1}}}+ \sum\limits_{\{ J\in \overline{T}_p\mid J+L> L\}}\sum\limits_{r\geq 1}\frac{\overline{ t^J v^2_{r,L,J}}}{p^{2r}},
\end{cases}\]for $u_{r,I,J},$  $v_{r,I,J}\in \F[x]_{< d}$.  Now $ \sum\limits_{M\in ({T}_p)_m}\overline{f_M\frac{\dd t_M}{t_M}}$ is done by the case $l=0$.\\

%Clearly by  (ii) of Definition \ref{Up}, $ {\overline{a_L\frac{\dd t_L}{t_L}}}\in U_p$ for each $L\in (\overline{T}_p)_m$. For each $K\in ({T}_p)_m\setminus(\overline{T}_p)_m$, we have %:

%\begin{align*}
   % \overline{ a_K\frac{\dd t_K}{t_K}}&= \sum_{\{J\in \overline{T}_p\hspace{1mm}|\hspace{1mm} J+K> K\}} \sum_{r\geq 1} \overline{\frac{ \overline{t^J v^2_{r,K,J}}}{p^{2r}}\frac{\dd t_K}{t_K}}%\\&= \sum_{\{J\in \overline{T}_p\hspace{1mm}|\hspace{1mm} J+K> K\}} \sum_{r\geq 1} \overline{\frac{ \overline{t^J v^2_{r,K,J}}}{p^{2r}}\frac{\dd t_{K'}}{t_{K'}}\wedge \frac{\dd p}{p}},\quad \text{(where $t^K =t^{K'}p$)}\\&= \sum_{\{J\in \overline{T}_p\hspace{1mm}|\hspace{1mm} J+K'> K'\}} \sum_{r\geq 1} \overline{\frac{ \overline{t^J v^2_{r,K,J}}}{p^{2r}}\frac{\dd t_{K'}}{t_{K'}}\wedge \frac{\dd p}{p}},\quad \text{( $J+K> K$ implies that $J+K'> K'$)}\\&\in U_p.
%\end{align*} 

By Definition \ref{Up2},  $\overline{a_L\frac{\dd t_L}{t_L}}\in U_p$  for each $L\in (\overline{T}_p)_m$.
Moreover, by the similar  way  of  (i) we can prove that for each $K\in (T_p)_m\setminus (\overline{T}_p)_m$,   $\overline{a_K\frac{\dd t_K}{t_K}}\in U_p$.
 Hence, for any  $l\geq 1$, $h_l\in \F[x]_{< d}$ and $I\in ({T}_p)_m$, we have : $$ \overline{ \frac{h_l}{p^l} \frac{\dd t_I}{t_I}}\in U_p+L_0+L_0\wedge \overline{\frac{\dd x}{x}}+L_0\wedge \overline{\frac{\dd p}{p}}+L_0\wedge \overline{\frac{\dd x}{x}\wedge \frac{\dd p}{p}}.$$
Therefore, we prove  $ \tilde{S}_p+\tilde{S}_p\wedge \overline{\frac{\dd p}{p}}\subseteq  U_p+L_0+L_0\wedge \overline{\frac{\dd x}{x}}+L_0\wedge \overline{\frac{\dd p}{p}}+L_0\wedge \overline{\frac{\dd x}{x}\wedge \frac{\dd p}{p}}$.

(iv) Suppose that $p$ is inseparable of degree $d$. By \cite[Proposition 3.8(3)]{ex}, we have:
         
         $$ S_p+S_p\wedge  \overline{\frac{\dd x}{x}}+L_0+L_0\wedge \overline{\frac{\dd x}{x}}= \tilde{S}_p+\tilde{S}_p\wedge \overline{\frac{\dd p}{p}}+L_0+L_0\wedge \overline{\frac{\dd x}{x}}.$$ Now using the previous part (iii),  we get:

         $$S_p+S_p\wedge  \overline{\frac{\dd x}{x}}\subseteq  U_p+L_0+L_0\wedge \overline{\frac{\dd x}{x}}+L_0\wedge \overline{\frac{\dd p}{p}}+L_0\wedge\overline{\frac{\dd x}{x}\wedge \frac{\dd p}{p}}. $$
         
(v) Recall from Theorem \ref{T23} that each element $\phi\in W_1(H_2^{m+1}(F_p))$ can be expressed as \begin{align*}\phi&= \sum_{K\in ({T}_p)_m\setminus (\overline{T}_p)_m \atop J\in \overline{T}_p, \,J+K> K} \sum_{r\geq 1}\overline{ \frac{\overline{t^Jv^2_{r,K,J}}}{p^{2r}}\frac{\dd t_K}{t_K}}+ \sum_{L\in (\overline{T}_p)_m, J\in {T}_p\atop t^J=t^{J'}p}\sum_{r\geq 0}\overline{ \frac{\overline{ t^{J'}u^2_{r,L,J}}}{p^{2r+1}}\frac{\dd t_L}{t_L}}\;+\\&\sum_{L\in (\overline{T}_p)_m, J\in \overline{T}_p\atop J+L> L}\sum_{r\geq 1}\overline{ \frac{\overline{ t^{J}v^2_{r,L,J}}}{p^{2r}}\frac{\dd t_L}{t_L}}+ \phi'\wedge \overline{\frac{\dd p}{p}}, \end{align*} for unique $u_{r,I,J}, v_{r,I,J}\in \F[x]_{<\deg p}$. Note that  the second and third sums of right side are the generators of $S'_{p,r}$ \textcolor{black}{for some $r\geq 1$}.  For the first sum of right side, since $J+K> K$ and $J_{\n}=0, K_{\n}=1$ , it follows that $J+K'> K'$, where $t^K=t^{K'}p$. Hence, these elements are the second type generators of $S'_{p,r}$ for $r$ even.   The uniqueness of $u_{r,I,J}$ and $v_{r,I,J}$ implies that $U_p=\bigoplus\limits_{r\geq 1} S'_{p,r}$. Consequently, we have $U_p^0=\bigoplus\limits_{r\geq 1} S_{p,r}^0$ when $p$  \textcolor{black}{is inseparable}.

Now suppose $p=\frac{1}{x}$. Note that the generators of $S'_{\frac{1}{x},r} $ appear in the expression of the form $\psi$ given in \cite[Theorem 4.2(2)]{ex}. Moreover, $\psi$ is uniquely decomposed there. Hence, 
$U_{\frac{1}{x}}=\bigoplus_{r\geq 1} S'_{\frac{1}{x},r}$ is a direct sum.

(vi) If   $p=\frac{1}{x}$,  then this is a consequence of \cite[Theorem 4.2(2)]{ex}. For other $p,$ it comes from  Theorem \ref{T23}.
     \end{proof}

We define some maps to introduce the  transfer maps $s_{p}^*$, where $p$ is a monic irreducible polynomial of degree $d$ or $p=1/x$:
\medskip

(1) Let $\theta_p: W_1(H_2^{m+1}(F_p))\xrightarrow{} H_2^m(\F(p))$ be the map given as follows: For every $\phi\in W_1(H_2^{m+1}(F_p))$, let 
$\theta_p(\phi)$ be the unique element of $H_2^m(\F(p))$ such that $ \phi\equiv  \theta_p(\phi)\wedge \overline{\frac{\dd p}{p}}\pmod{\overline{U_p}}.$

(2) For $p\neq 1/x$, let $t_p: \F(p)\xrightarrow{}\F$ be the $\F$-linear map defined by: $ t_p(x^i)=0$ for $0\leq i\leq d-2$, and $t_p(x^{d-1})=1$. Let $t_p': W_q(\F(p))\xrightarrow{} W_q(\F)$ be the Scharlau transfer induced by $t_p$ \cite[Section 20]{El}\label{transfer}.
Now by Lemma  \ref{L41},  we have $ t_p'( I_q^m(\F(p)))\subset I_q^m(\F)$, for all $m\in \N.$  So we obtain a  map $t_p'': \overline{I}_q^{m}(\F(p))\xrightarrow{} \overline{I}_q^{m}(\F)$ induced from $t_p'$, for all $m\in\N$. Moreover,  $t_p''$ induces   the  map $t_p^* : H_2^{m}(\F(p))\xrightarrow{} H_2^{m}(\F)$ defined by \[  t_p^*= {f(\F)}_{m}^{-1}\circ t_p''\circ{f(\F(p))}_{m},\] where $  f(\F(p))_{m}$, $f(\F)_m$ are the  Kato's isomorphisms given in (\ref{Ke}) over $\F(p)$ and $\F$, respectively.

(3) For $p=1/x$, we take $t_p^*$ the identity since $\F(p)$ is isomorphic to $\F$.\qed
\medskip

\noindent{\bf Notation.} For each $h\in \F[x]$, let $h_c$ denote its constant part.

\begin{defn}\label{dt} {\bf (Transfer maps.)}

\noindent{(i)} Suppose that $p$ is  separable or $p=\frac{1}{x}$, then the transfer $s_p^*: W_1(H_2^{m+1}(F_p))\xrightarrow{} H_2^m(\F)$ is  defined by:$$s_p^*(\phi)= t_p^*(\theta_p(\phi))$$for $\phi \in W_1(H_2^{m+1}(F_p))$. Note that $s_{\frac{1}{x}}^*$ is nothing but the map $\eta$ defined in \cite[Theorem 4.9]{ex}.

\noindent{(ii)} Suppose that $p$ is inseparable. By Theorem \ref{T23},  we can write $ \phi-\theta_p(\phi)\wedge \overline{\frac{\dd p}{p}}$ modulo $\overline{U_p^0}$, uniquely as follows:\begin{align*} \textcolor{black}{\sum\limits_{I\in (\overline{T}_p)_m\atop J\in \overline{T}_p,  I_x=1}\sum\limits_{i\geq 0}\overline{ \frac{\overline{ t^J s_{{2i+1},I,J}^2}}{p^{2i+1}}\frac{\dd t_I}{t_I}}\;+\sum\limits_{I\in (\overline{T}_p)_m\atop  I_x=1}\sum_{J\in \overline{T}_p\atop J+I> I}\sum\limits_{i\geq 1}\overline{ \frac{\overline{ t^J s_{2i,I,J}^2}}{p^{2i}}\frac{\dd t_I}{t_I}}}\\+\textcolor{black}{\sum\limits_{I\in (\overline{T}_p)_{m-1}\atop I_x=1}\sum_{J\in \overline{T}_p\atop J+I> I} \sum\limits_{i\geq 1} \overline{ \frac{\overline{t^Js^2_{{2i},I,J}}}{p^{2i}}\frac{\dd t_I}{t_I}\wedge \frac{\dd p}{p}},}  \end{align*} where $ s_{2i,I,J}, s_{2i+1,I,J}\in \F[x]_{< d}$. For each $I\in (\overline{T}_p)_m$ with $I_x=1$, define $I'\in (\overline{T}_p)_{m-1}$ such that $t^{I}=t^{I'}x$ for all $m\geq 1$.
Then, the transfer  $s_p^*: W_1(H_2^{m+1}(F_p))\xrightarrow{} H_2^m(\F) $ is a homomorphism defined by \textcolor{black}{ \begin{align*}  s_p^*(\phi)&= \sum\limits_{I\in (\overline{T}_p)_m\atop J\in \overline{T}_p,  I_x=1}\sum\limits_{i\geq 0}\overline{ \frac{\left(\overline{ t^J s_{2i+1,I,J}^2}\right)_c}{p_c^{2i+1}}\frac{\dd t_{I'}}{t_{I'}}}+\sum\limits_{I\in (\overline{T}_p)_m\atop I_x=1}\sum_{J\in \overline{T}_p\atop J+I> I}\sum\limits_{i\geq 1}\overline{ \frac{\left(\overline{ t^J s_{2i,I,J}^2}\right)_c}{p_c^{2i}}\frac{\dd t_{I'}}{t_{I'}}}\\&+\sum\limits_{I\in (\overline{T}_p)_{m-1}\atop I_x=1}\sum_{J\in \overline{T}_p\atop J+I> I} \sum\limits_{i\geq 1} \overline{ \frac{\left(\overline{t^Js^2_{2i,I,J}}\right)_c}{p^{2i}_c}\frac{\dd t_{I'}}{t_{I'}}\wedge \frac{\dd p_c}{p_c}} +t_p^*(\theta_p(\phi)). \end{align*}}
\end{defn}

\section{The reciprocity law for $ L_0+ L_0\wedge \overline{\frac{\dd p}{p}}$}\label{S4}
Our aim in this section is to prove the reciprocity law for $L_0+L_0\wedge \overline{\frac{\dd p}{p}}$, which means ${\oplus_{q}s^*_{q}( \partial_q(\phi))=0}$ for all $\phi\in L_0+L_0\wedge \overline{\frac{\dd p}{p}}$. Moreover, when $p$ is inseparable, we extend this law  for   $L_0\wedge \overline{\frac{\dd x}{x}\wedge \frac{\dd p}{p}}$.

First of all we introduce some notations. Let $ p=\sum_{i=0}^{d} p_ix^{d-i}$ be a monic irreducible polynomial ($p_0=1$). From now on, we work in $\F(x)$ and $\F(p)$ and use $x$ to denote the  elements of both fields.  Let us define $\gamma_i\in \F$ such that$$x^{d+i-1}= \gamma_ix^{d-1}+ g_i\in \F(p),$$where  $\deg g_i\leq d-2$. It is clear that $ \gamma_0=1$ and $\gamma_1= p_1.$  By \cite[Section 5]{AJ}, we have the relation \begin{equation}\label{g1}
     \gamma_i= \gamma_{i-1}p_1+\gamma_{i-2}p_2+\ldots+ \gamma_0p_i\quad \text{ for $1\leq i\leq d$}.
 \end{equation}
For $i>d,$ we have  \begin{align*}
    x^{d+i-1}&=p_1x^{d+i-2}+p_2x^{d+i-3}+\ldots+ p_dx^{i-1}\in \F(p)\\&= p_1( \gamma_{i-1}x^{d-1}+g_{i-1})+p_2(\gamma_{i-2}x^{d-1}+g_{i-2})+\ldots + p_d(\gamma_{i-d}x^{d-1}+g_{i-d})\\&= ( p_1\gamma_{i-1}+p_2\gamma_{i-2}+\ldots +p_d\gamma_{i-d})x^{d-1}+G,
\end{align*}
where $G$ is the sum of the remaining terms, so $\deg G\leq d-2.$ Therefore, \begin{equation}\label{g2}
\gamma_{i}= \gamma_{i-1}p_1+\gamma_{i-2}p_2+\ldots+ \gamma_{i-d}p_d\quad \text{ for $i>d$}.
    \end{equation}
Moreover, we take $\gamma_i=0$ for $i<0.$ Therefore,  by (\ref{g1}) and (\ref{g2}) we have : \begin{equation}\label{g3}
    \gamma_i=\sum_{k=1}^{d} \gamma_{i-k}p_k\quad \text{for all $i\geq 1$.}
\end{equation}
    From  $\gamma_i$, we get the relation:
    \begin{equation}\label{re1}
        (1+p_1x+p_2x^2+\ldots+p_dx^d)(1+\gamma_1x+\gamma_2x^2+\ldots)=1\in \F((x)).
    \end{equation}
    Similarly, we also have:
    \begin{equation}\label{re}
        (1+p_1x^{-1}+p_2x^{-2}+\ldots +p_dx^{-d})(1+\gamma_1x^{-1}+\gamma_2x^{-2}+\ldots)=1\in \F(x)_{\frac{1}{x}}.
    \end{equation}
    Moreover, $p= x^d(1+p_1x^{-1}+\ldots +p_dx^{-d})$. 
    Therefore, from (\ref{re}) we get:
    \begin{equation}\label{pi}
        p^{-1}= x^{-d}( 1+\gamma_1x^{-1}+\gamma_2x^{-2}+\ldots )\in \F(x)_{\frac{1}{x}}.
    \end{equation}

We prove the following lemma, which directs the well-\textcolor{black}{definedness} of the transfer map $t_p^*$, for $p$ monic irreducible.
\vspace{2mm}

\begin{lemma}\label{L41} Let $p$ be a monic irreducible polynomial, and $t_p,$ $t_p'$ as defined in Page \pageref{transfer}. Then, we have $t_p'( I_q^m(\F(p)))\subseteq I_q^m(\F)$ for all $m\geq 1$.
\end{lemma}

\begin{proof}
We treat two cases depending on  the separability of $p$.\\
    
     \noindent{Case 1 :} If $p$ is separable, then we are done by \cite[Corollary 2.5]{mu}.
     
\noindent{Case 2 :} Suppose $p=\sum\limits_{i=0}^{d}p_{i}x^{d-i}$ is inseparable ($p_0=1$). Using  the  three relations in $W_q(\F(p))$:
\begin{itemize}
\item[(a)] $[1,h_1]+[1,h_2]=[1,h_1+h_2]$,
\item[(b)] $\langle \langle f+g ; h]]= \langle \langle f ;{fh}{(f+g)}^{-1}]]+\langle\langle g ;{gh}{(f+g)}^{-1}]]$,
\item[(c)] $\langle\langle ax,bx\rangle\rangle=\langle \langle ab,ax\rangle\rangle$,
\end{itemize}\noindent{we easily see that $I^m_q(\F(p))$ is {additively} generated by  the $m$-fold quadratic  Pfister forms ${ \langle \langle a_1,a_2,\ldots, a_{m-2}, a_{m-1}x^l ; a_mx^k]]},$ where  $l\in \{0,1\}$, $0\leq k\leq d-1$, $a_i\in \F^*$ for $1\leq i\leq m-1$ and $a_m\in \F$.}

Combining the Frobenius reciprocity \cite[Proposition 20.2]{El} with \cite[Corollary 21.5]{El}, we reduce to prove that $t'_p(\langle \langle ax^l ; bx^k]])\in I^2_q(\F)$ for every $a\in \F^*$ and $b\in \F$, when $1\leq k\leq d-1$ and $l=1$. In fact, we have \textcolor{black}{ $\langle \langle ax; bx^k]]=[1, bx^k] + a[x, bx^{k-1}]$}, and we will compute the transfer separately for the $2$-dimensional forms.

As $p$ is inseparable, $d$ is even and $p_i=0$ for $i$ odd. Let \textcolor{black}{us} take $d=2e$ for a suitable $e\geq 1$. 
        By \cite[Lemma 5.2]{AJ}, we have
        \begin{align*}
            t_p'([1,bx^k])&= \sum\limits_{i=1}^{e-1}[p_{2(e-i)+1}, b\gamma_{k+2i-d-1}]+[\gamma_1, b\gamma_{k-1}]\\ &= [\gamma_1, b\gamma_{k-1}]\quad \textcolor{black}{(\text{since }  p_j=0, \text{ for } j\hspace{1mm} \text{ odd})}\\&=0\in W_q(\F)\quad \textcolor{black}{(\text{since } \gamma_1=p_1=0)}.
        \end{align*}

Now we compute $t_p'([x, bx^{k-1}])$ for $b\in \F$ and $1\leq k\leq d-1$. Recall that $\gamma_1=p_1=0$ and $\gamma_2= p_1\gamma_1+p_2\gamma_0=p_2.$ By \cite[Lemma 5.2]{AJ}, we have 
 \begin{align*}
            t_p'([x, bx^{k-1}])&= [1, b(\gamma_k+ p_1^2\gamma_{k-2})]+ \sum_{i=1}^{e-1}[p_{2(e-i)+2}, b\gamma_{k-1+2i-d-1}]+[\gamma_2, b\gamma_{k-2}]\\&=  [1,b\gamma_k]+\sum_{i=1}^{e}[p_{2i}, b\gamma_{k-2i}]\quad \textcolor{black}{(\text{as } d=2e, p_1=0, p_2=\gamma_2)}
            \\&=[1,b\gamma_k]+ \sum_{i=1}^{k}[p_i,b\gamma_{k-i}]\quad \textcolor{black}{(\text{as } p_i=0 \text{ for i odd and }  \gamma_i=0 \text{ for } i<0) }  \\&= [1,b\sum_{i=1}^{k} p_i \gamma_{k-i}]+ \sum_{i=1}^{k}[p_i,b\gamma_{k-i} ]\quad \textcolor{black}{(\text{by } (\ref{g1}))}\\&= \sum_{i=1}^{k} [1,bp_i\gamma_{k-i}]+\sum_{i=1}^{k}[p_i,b\gamma_{k-i} ]= \sum_{i=1}^{k} \langle 1,p_i\rangle \otimes [1,bp_i\gamma_{k-i}]\in I^2_q(\F).
\end{align*}This finishes the proof of the lemma.
\end{proof}

The following lemma will be  needed to prove the required  reciprocity law given in Lemma \ref{L44}.
\begin{lemma}\label{L42}Let  $p=\sum\limits_{i=0}^{d}p_ix^{d-i}$ be a monic irreducible polynomial of degree $d$, and $t_p'$ the transfer map as defined in Page \pageref{transfer}.\\(i) If $i\geq 0$ and $a\in \F$, then ${t_p'([1,ax^i])=\begin{cases}\sum\limits_{j=0}^{e-1}[p_{2j+1},a\gamma_{i-2j-1}] & \text{if}\;\; d=2e,
        \\ \sum\limits_{j=0}^{e} [p_{2j}, a\gamma_{i-2j}] & \text{if}\;\;d=2e+1.
            \end{cases}}$\\(ii) If $i\geq 1$ and $a\in \F$, then $  t_p'(\langle\langle x;ax^i]])=
                \sum\limits_{j=1}^{d}\langle\langle p_j;ap_j\gamma_{i-j}]].$\\(iii) If $p$ is inseparable, then  $ t_p'( \langle\langle x;a]])= \langle\langle p_d;a]]$ for any $a\in \F$.
    
\end{lemma}
\begin{proof} (i) Let $i\in \N$ and $a\in \F$. Now we have two cases depending on $d$ is even or odd.
    
\noindent{Case 1:} Suppose $d=2e$ for some $e\geq 1$. By \cite[Lemma 5.2]{AJ},  we have
\begin{align*}
t_p'([1,ax^i])&= \sum\limits_{l=1}^{e-1}[p_{2(e-l)+1}, a\gamma_{i+2l-d-1}]+[\gamma_1, a\gamma_{i-1}] \\&= \sum\limits_{l=1}^{e-1}[p_{2(e-l)+1}, a\gamma_{i+2l-2e-1}] +[p_1,a\gamma_{i-1}]\quad \textcolor{black}{ (\text{putting  }  d=2e, \gamma_1=p_1)}\\&= \sum\limits_{j=0}^{e-1}[ p_{2j+1}, a\gamma_{i-2j-1}].
\end{align*} 

\noindent{Case 2:} Suppose $d=2e+1$ for some $e\geq 0$. By \cite[Lemma 5.2]{AJ},  we have
\begin{align*}
            t_p'([1,ax^i])&=\sum\limits_{l=0}^{e}[p_{2(e-l)}, a\gamma_{i+2l-d+1}]\\&= \sum\limits_{l=0}^{e} [p_{2(e-l)}, a\gamma_{i-2(e-l)}] \quad \textcolor{black}{(\text{putting  } d=2e+1)}\\&= \sum\limits_{j=0}^{e}[p_{2j}, a\gamma_{i-2j}].
        \end{align*}

\noindent{(ii)} Let $ i\geq 1$ and $a\in \F$.  As above, we have two cases depending on $d$ is even or odd.
        
        \noindent{Case 1:} Let $d=2e$ be even for some $e\geq 1$. Now $\langle\langle x;ax^i]]= [1,ax^i]+[ x,ax^{i-1}] $.  \textcolor{black}{Therefore, using \cite[Lemma 5.2]{AJ}}, we have:
            \begin{align*}
                t_p'(\langle\langle x;ax^i]])&= t_p'([1,ax^i])+t_p'([x,ax^{i-1}])\\&= \left(\sum_{j=1}^{e-1}[p_{2(e-j)+1}, a\gamma_{i+2j-2e-1}]+[\gamma_1,a\gamma_{i-1}]\right) \\& +\left([1,a(\gamma_{i}+p_1^2\gamma_{i-2})]+\sum_{j=1}^{e-1}[p_{2(e-j)+2}, a\gamma_{i+2j-2e-2}]+[\gamma_2,a\gamma_{i-2}]\right)\\&=\sum_{j=0}^{2e}[p_j,a\gamma_{i-j}]+[1,ap_1^2\gamma_{i-2}]+[p_1^2, a\gamma_{i-2}]\quad \textcolor{black}{(\text{putting } \gamma_2= p_1^2+p_2)}\\&= \sum_{j=0}^{2e}[p_j,a\gamma_{i-j}]\quad \textcolor{black}{(\text{since } [p_1^2,a\gamma_{i-2}]=[1,ap_1^2\gamma_{i-2}])}.
            \end{align*}
 \noindent{Case 2:} Let $d=2e+1$ be odd for some $e\geq 0$. Again using \cite[Lemma 5.2]{AJ} we have 
             \begin{align*}
                t_p'(\langle\langle x;ax^i]])&= t_p'([1,ax^i])+t_p'([x,ax^{i-1}])\\&= \sum_{j=0}^{e}[p_{2(e-j)}, a\gamma_{i+2j-2e}]+\sum_{j=0}^{e}[p_{2(e-j)+1}, a\gamma_{i-1+2j-2e}]\\&= \sum_{j=0}^{2e+1}[p_j,a\gamma_{i-j}].
                \end{align*} Combining the two cases, we get ${ t_p'(\langle\langle x;ax^i]])= \sum\limits_{j=0}^{d} [p_j,a\gamma_{i-j}]}$.  Moreover, (\ref{g3})  says that $ \gamma_i=\sum\limits_{j=1}^{d}p_j\gamma_{i-j}$. Consequently, $\sum\limits_{j=0}^{d} p_j\gamma_{i-j}=0$. Therefore, we have  \begin{align*} t_p'( \langle \langle x;ax^i]])=\sum\limits_{j=0}^{d} [p_j,a\gamma_{i-j}]+ \left[1,\sum\limits_{j=0}^{d}ap_j\gamma_{i-j}\right]&=\sum\limits_{j=0}^{d} [p_j,a\gamma_{i-j}]+\sum_{j=0}^{d}[1, ap_j\gamma_{i-j}]\\&= \sum\limits_{j=1}^{d}\langle\langle p_j;ap_j\gamma_{i-j}]]\quad \textcolor{black}{(\text{since }  p_0=1)}.\end{align*} Thus, statement (ii) follows.
\medskip

\noindent{(iii)} Suppose that $p$ is inseparable. Then, $d$ is even and $p_i=0$ when $i$ is odd. Let $d=2e$ for some $e\geq 1$. Now $ \langle\langle x;a]]=[1,a]+[x,ax^{-1}]$. {We have seen in the proof of Lemma \ref{L41} that $t_p'([1,a])=0$.} Therefore, we have to prove that $ t_p'([x,ax^{-1}])=\langle\langle p_d;a]]$. Now $p=\sum\limits_{i=0}^{d}p_ix^{d-i}=0\in \F(p)$, where $p_0=1$.
Hence,  $x^{-1}= p_d^{-1}\left( \sum\limits_{i=0}^{d-1}p_ix^{d-(i+1)}\right)\in \F(p)$. Consequently, we have     \begin{align*}
    t_p'([x,ax^{-1}])= t_p'\left(\left[x,\sum_{i=0}^{d-1} ap_ip_d^{-1}x^{d-i-1}\right]\right)=  \sum_{i=0}^{d-1} t_p'\left(\left[ x,ap_ip_d^{-1}x^{d-i-1}\right]\right)\end{align*} \textcolor{black}{By \cite[Lemma 5.2]{AJ} and using $p_1=0$, we get \begin{align*}t_p'([x,ax^{-1}])&= \sum_{i=0}^{d-1}\left\{ \sum_{j=1}^{e-1}[ p_{2(e-j)+2}, ap_ip_d^{-1}\gamma_{2j-i-2}]+[\gamma_2,ap_ip_d^{-1}\gamma_{d-i-2}]\right\} \\&{+\sum_{i=0}^{d-1} [1,ap_ip_d^{-1} \gamma_{d-i}] }\\&= \underbrace{ \sum_{i=0}^{d-1} \sum_{j=1}^{e-1}[ p_{2(e-j)+2}, ap_ip_d^{-1}\gamma_{2j-i-2}]}_{B\textcolor{black}{\text{ (say)}}}+ \left[1,ap_d^{-1}\sum_{i=0}^{d-1}p_i\gamma_{d-i}\right]\\&+ \left[\gamma_2, ap_d^{-1}\sum_{i=0}^{d-1}p_i\gamma_{d-i-2}\right]\\&= [1,a]+ B.
\end{align*}}
The last equality holds because by (\ref{g3}), $ \gamma_d=\sum\limits_{i=1}^{d} p_i\gamma_{d-i}$, which implies that $ \sum\limits_{i=0}^{d-1}p_i\gamma_{d-i}=p_d\gamma_0=p_d.$ Similarly, using the recurrence relation for $\gamma_{d-2}$, we get $\sum\limits_{i=0}^{d-1}p_i\gamma_{d-i-2}=0 $. Now  we compute $B$:
\begin{align*}
    B&=\sum_{i=0}^{d-1} \sum_{j=1}^{e-1}[ p_{2(e-j)+2}, ap_ip_d^{-1}\gamma_{2j-i-2}]\\&= \sum_{j=1}^{e-1}\left[ p_{2(e-j)+2}, ap_d^{-1}\sum_{i=0}^{ 2j-2}p_i\gamma_{2j-i-2}\right] \quad \textcolor{black}{(\text{using } \gamma_i=0 \text{ for } i<0)}\\&= [p_{2e}, ap_d^{-1}p_0\gamma_0]+ \sum_{j=2}^{e-1}\left[ p_{2(e-j)+2}, ap_d^{-1}\sum_{i=0}^{2j-2}p_i\gamma_{2j-i-2}\right]\\&=[p_d, ap_{d}^{-1}]\quad \textcolor{black}{\left(\text{since }   \gamma_{2j-2}=\sum_{i=1}^{2j-2}p_i\gamma_{2j-2-i}, \text{ for all } j\in \{2,\ldots,e-1\}\right)}.
\end{align*}
 Therefore, $t_p'([x,ax^{-1}])=[1,a]+[p_d, ap_d^{-1}]=\langle\langle p_d;a]]$.
 Hence (iii)  is proved.
        \end{proof}
    In the next lemma   we  compute $s_{\frac{1}{x}}^*\circ \partial_{\frac{1}{x}}(\phi)$,  where $\phi\in L_0\wedge \overline{\frac{\dd p}{p}}$ and  $p=\sum_{i=0}^{d}p_ix^{d-i}$ is a  monic  irreducible  polynomial of degree $d$.

\begin{lemma}\label{L43}
    (i)  If ${\phi= \overline{ax^k \frac{\dd t_I}{t_I}\wedge \frac{\dd p}{p}}}$, where $a\in \F, k\geq 0$ and  $I\in (T)_{m-1}$, then \[s_{\frac{1}{x}}^*\circ\partial_{\frac{1}{x}}(\phi)=\begin{cases}
            {\sum\limits_{j=0}^{e-1} \overline{ap_{2j+1}\gamma_{k-2j-1}\frac{\dd t_I}{t_I}}} & \text{if}\;\; d=2e,\\
            {\sum\limits_{j=0}^{e} \overline{ap_{2j}\gamma_{k-2j}\frac{\dd t_I}{t_I}}} & \text{if}\;\; d=2e+1.
        \end{cases}\]
(ii) If ${\phi= \overline{ ax^k\frac{\dd t_I}{t_I}\wedge\frac{\dd x}{x}\wedge  \frac{\dd p}{p}}}$, where $ a\in \F$, $k\geq 1$ and  $I\in (T)_{m-2}$, then
        $ {s_{\frac{1}{x}}^*\circ \partial_{\frac{1}{x}} (\phi)=  \sum\limits_{i=1}^{d} \overline{ ap_i\gamma_{k-i}\frac{\dd t_I}{t_I}\wedge \frac{\dd p_i}{p_i}}}.$
    \end{lemma}

\begin{proof}
    (i) Recall that  $p= \sum\limits_{i=0}^{d} p_ix^{d-i}\in \F[x]$ with $p_0=1.$ Here, we use the 
 following fundamental relation over any field:
\begin{equation}\label{feq}
 \frac{\dd (u+v)}{u+v}= \frac{u}{u+v}\frac{\dd u}{u}+\frac{v}{u+v}\frac{\dd v}{v},     
 \end{equation} where $u+v\neq 0$. Now we have 
        \begin{align*}
        \phi= \overline{ax^k\frac{\dd t_I}{t_I}\wedge \frac{\dd p}{p}}&= \sum\limits_{i=0}^{d} \overline{\frac{ p_ix^{d-i}ax^k}{p}\frac{\dd t_I}{t_I}\wedge \frac{\dd (p_ix^{d-i})}{p_ix^{d-i}}}\quad (\textcolor{black}{\text{by } (\ref{feq})})\\ &= \underbrace{\sum\limits_{i=0}^{d} \overline{\frac{p_ix^{d-i}ax^k}{p} \frac{\dd t_I}{t_I}\wedge \frac{\dd p_i}{p_i}}}_{A}+ \underbrace{\sum\limits_{\textcolor{black}{d-i \text{ odd}}} \overline{\frac{p_ix^{d-i}ax^k}{p} \frac{\dd t_I}{t_I}\wedge\frac{\dd x}{x}}}_{B}.
        \end{align*}
        Moreover, over $\F(x)_{\frac{1}{x}}$, we 
 obtain \begin{align*}B&=\sum\limits_{d-i \hspace{1mm} odd} \overline{\frac{p_ix^{d-i}ax^k}{p} \frac{\dd t_I}{t_I}\wedge\frac{\dd x}{x}} \\&=\sum\limits_{d-i \hspace{1mm} odd} \overline{{ap_ix^{k-i}(1+\gamma_1x^{-1}+\gamma_2x^{-2}+\ldots)}\frac{\dd t_I}{t_I}\wedge\frac{\dd x}{x}}\quad \textcolor{black}{(\text{by  (\ref{pi})})}\\&= \sum\limits_{d-i \hspace{1mm} odd} \overline{(ap_i\gamma_{k-i}+ R_i) \frac{\dd t_I}{t_I}\wedge\frac{\dd x}{x}}\quad \textcolor{black}{(\text{by Remark  \ref{R21})}}, \end{align*} where  each $R_i\in x\F[x]$.
        Now we have two cases, $d$ is even or odd. \\
        
            \noindent{Case 1 :}  Let  $d=2e$ be even, where $e\geq 1$. Then we have  \begin{align*}
                s_{\frac{1}{x}}^*(\partial_{\frac{1}{x}}(\phi))&= s_{\frac{1}{x}}^*(\partial_{\frac{1}{x}}(B))\quad \textcolor{black}{ (\text{since }   s_{\frac{1}{x}}^*(\partial_{\frac{1}{x}}(A))=0)}.\end{align*} \textcolor{black}{Note that $ s_{\frac{1}{x}}^*=\eta$, the form given in \cite[Theorem 4.9]{ex}. Therefore, we  get \begin{align*} s_{\frac{1}{x}}^*(\partial_{\frac{1}{x}}(\phi))= s_{\frac{1}{x}}^*(\partial_{\frac{1}{x}}(B))= \sum\limits_{d-i \textcolor{black}{\text{ odd}}} \overline{ap_i\gamma_{k-i} \frac{\dd t_I}{t_I}} =\sum\limits_{j=0}^  {e-1} \overline{a p_{2j+1}\gamma_{k-{2j-1}} \frac{\dd t_I}{t_I}}. \end{align*}  }
  \noindent{Case 2 :} Let $d=2e+1$ be odd,  for some $e\geq 0$. Here also we get\begin{align*}
                s_{\frac{1}{x}}^*(\partial_{\frac{1}{x}}(\phi))= s_{\frac{1}{x}}^*(\partial_{\frac{1}{x}}(B))= \sum\limits_{d-i \textcolor{black}{\text{ odd}}} \overline{ap_i\gamma_{k-i} \frac{\dd t_I}{t_I}} =\sum\limits_{j=0}^  {e} \overline{a p_{2j}\gamma_{k-{2j}} \frac{\dd t_I}{t_I}}. \end{align*} 
            Thus, statement (i) is proved.\\
            
\noindent{(ii)}  For  any $a\in \F$, $k\geq 1$ and $I\in (T)_{m-2}$, we have  :\begin{align*}
        \phi= \overline{ax^k\frac{\dd t_I}{t_I}\wedge \frac{\dd x}{x}\wedge \frac{\dd p}{p}}&= \sum\limits_{i=0}^{d} \overline{\frac{ p_ix^{d-i}ax^k}{p}\frac{\dd t_I}{t_I}\wedge \frac{\dd x}{x}\wedge \frac{\dd (p_ix^{d-i})}{p_ix^{d-i}}}\quad (\textcolor{black}{\text{by (\ref{feq})}})\\ &=\sum\limits_{i=0}^{d} \overline{\frac{ap_ix^{d+k-i}}{p} \frac{\dd t_I}{t_I}\wedge \frac{\dd p_i}{p_i}\wedge \frac{\dd x}{x}}\\&=\sum\limits_{i=1}^{d} \overline{\frac{ap_ix^{d+k-i}}{p} \frac{\dd t_I}{t_I}\wedge \frac{\dd p_i}{p_i}\wedge \frac{\dd x}{x}} \quad \textcolor{black}{(\text{as  } p_0=1)}.
        \end{align*}
         {Repeating the same process done in statement (i),} we get  over  $\F(x)_{\frac{1}{x}}$ 
         \[ \phi= \sum\limits_{i=1}^{d} \overline{ (ap_i\gamma_{k-i}+R_i) \frac{\dd t_I}{t_I}\wedge \frac{\dd p_i}{p_i}\wedge \frac{\dd x}{x}},\]where each $R_i\in x\F[x].$ Hence, \[ s_{\frac{1}{x}}^*(\partial_{\frac{1}{x}}(\phi))= \sum\limits_{i=1}^{d} \overline{ ap_i\gamma_{k-i} \frac{\dd t_I}{t_I}\wedge \frac{\dd p_i}{p_i}}.\]
         Thus, statement (ii) follows.
\end{proof}
Now we prove  the  reciprocity law of $L_0+L_0\wedge \overline{\frac{\dd p}{p}}$, and for inseparable $p$ of $L_0\wedge\overline{\frac{\dd x}{x}\wedge \frac{\dd p}{p}}$ .

\begin{lemma}\label{L44} Let $ p=\sum_{i=0}^{d}p_ix^{d-i}$ be a monic irreducible polynomial of degree $d$ over $\F$. Then, we have:\\(i) $\sum\limits_{q}s_q^*(\partial_q(\phi))=0$ for all $\phi\in L_0+L_0\wedge \overline{\frac{\dd p}{p}}$, where $q$ varies over  $\frac{1}{x}$ and all monic irreducible polynomials of $\F[x]$.\\(ii) If $p$ is inseparable, then $\sum\limits_{q}s_q^*(\partial_ q(\phi))=0$ for all $\phi\in L_0\wedge \overline{\frac{\dd x}{x}\wedge \frac{\dd p}{p}}.$
\end{lemma}

\begin{proof}
    (i)  Recall that  $L_0$ is  generated by these  two types of forms  $ {\overline{ ax^i\frac{\dd t_I}{t_I}}}$ and ${\overline{bx^j\frac{\dd t_J}{t_J}\wedge \frac{\dd x}{x}}} $, where  $I\in (T)_m,J\in (T)_{m-1}$, $a,b\in \F$ and $ i\geq 0, j\geq 1$.  Let  $\phi={\overline{ ax^i\frac{\dd t_I}{t_I}}} $,  where $a,i,I$ are as stated above. Then by \cite[Lemma 3.3]{ex},  $\partial_q(\phi)=0$ for all $q$ monic irreducible. 
         For  ${q=\frac{1}{x}},$  since $I\in (T)_m$, it follows from the  definition of $ s_{\frac{1}{x}}^*$ that  $$ {s_{\frac{1}{x}}^*(\partial_{\frac{1}{x}}(\phi))=0}.$$ 
        {Now suppose $\phi={\overline{bx^j\frac{\dd t_J}{t_J}\wedge \frac{\dd x}{x}}}$,} where $b\in \F,j\geq 1$ and $J\in (T)_{m-1}$. Again by \cite[Lemma 3.3]{ex},  $\partial_q(\phi)=0$, where  $q\neq x$ is a monic irreducible polynomial. 
Since $j\geq 1$,  Remark \ref{R21} implies that   $\partial_x(\phi)=0$ and  by definition of the transfer maps,  we have ${s_{\frac{1}{x}}^*(\partial_{\frac{1}{x}}(\phi))=0}.$ Therefore, we get \begin{equation}\label{Leq}
     \sum\limits_{q}s_q^*( \partial_q(\phi))=0, \quad \forall\; \phi\in L_0.
 \end{equation}
Note that ${L_0\wedge \overline{\frac{\dd p}{p}}} $ is generated by the two types of elements  $ {\overline{ ax^i\frac{\dd t_I}{t_I}\wedge \frac{\dd p}{p}}}$ and ${\overline{bx^j\frac{\dd t_J}{t_J}\wedge \frac{\dd x}{x}\wedge \frac{\dd p}{p}}} $  where $a,b\in \F$, $ i\geq 0, j\geq 1$ and $I\in (T)_{m-1},J\in (T)_{m-2}$. 
If $ {\phi=\overline{ ax^i\frac{\dd t_I}{t_I}\wedge \frac{\dd p}{p}}}$, then  \cite[Lemma 3.3]{ex} implies that   $\partial_q(\phi)=0$ for all  $q\neq p$ monic irreducible.

Now suppose $\phi={\overline{bx^j\frac{\dd t_J}{t_J}\wedge \frac{\dd x}{x}\wedge \frac{\dd p}{p}}}$, where $b\in \F$, $j\geq 1$ and $J\in (T)_{m-2}$. {Hence, Remark \ref{R21} implies that   $\partial_x(\phi)=0$.} Moreover, by \cite[Lemma 3.3]{ex}, $\partial_q(\phi)=0$ for all  monic irreducible $q$ except $p$ and $x$.

Thus, till now we proved that  for any  ${\phi\in L_0\wedge \overline{\frac{\dd p}{p}}}$, we have  $ {s_q^*(\partial_q(\phi))=0}$ for all $q$ except $p$ and $\frac{1}{x}$. Hence, it suffices to  prove that $$s_p^*(\partial_p(\phi))=s_{\frac{1}{x}}^*(\partial_{\frac{1}{x}}(\phi)$$for any generator $\phi$ of ${L_0\wedge \overline{\frac{\dd p}{p}} }$. Let us consider $ \phi={\overline{ ax^i\frac{\dd t_I}{t_I}\wedge \frac{\dd p}{p}}} $, where $a\in \F$, $i\geq 0$ and  $I\in (T)_{m-1}$, such that $ \tilde{I}=\{l_1,\ldots,l_{m-1}\}$ with $l_1<l_2<\ldots<l_{m-1}$. Therefore,  we get \begin{align*}
    s_p^*(\partial_p(\phi))&= s_p^*\left( \overline{ ax^i\frac{\dd t_I}{t_I}\wedge \frac{\dd p}{p}}\right)=t_p^*\left( \overline{ ax^i\frac{\dd t_I}{t_I}}\right)\\&= {f(\F)}^{-1}_m\left(t_p''\left(\overline{ \langle \langle t_{l_1},\ldots,t_{l_{m-1}};ax^i]]}\right)\right) \quad \textcolor{black}{(\text{by  the definition of } t_p^*)}\\&=  {f(\F)}^{-1}_m\left( \overline{ t_p'\left(\langle \langle t_{l_1},\ldots,t_{l_{m-1}};ax^i]]\right)}\right).\end{align*} \textcolor{black}{Now using \cite[Proposition 20.2]{El}, we have  \begin{align*} s_p^*(\partial_p(\phi))=f(\F)_m^{-1} \left(\overline{\langle \langle t_{l_1},\ldots,t_{l_{m-1}}\rangle\rangle\otimes t_p'([1,ax^i])}\right).
        \end{align*}} Recall that $p=\sum\limits_{i=0}^{d}p_ix^{d-i}$ with $p_0=1$. Putting the value of $t_p'([1,ax^i])$ from   Lemma \ref{L42}(i), we get  
        \begin{align*}
            s_p^*(\partial_p(\phi))&=\begin{cases}
                 f(\F)^{-1}_m\left(\sum\limits_{j=0}^{e-1} \overline{\langle\langle t_{l_1},\ldots, t_{l_{m-1}}; ap_{2j+1} \gamma_{i-2j-1}]]}\right) &  \textcolor{black}{\text{if}} \; d=2e,\\
                  f(\F)^{-1}_m\left(\sum\limits_{j=0}^{e} \overline{\langle\langle t_{l_1},\ldots, t_{l_{m-1}}; ap_{2j} \gamma_{i-2j}]]}\right) &  \textcolor{black}{\text{if}} \; d=2e+1.
            \end{cases}\\&= \begin{cases}
                {\sum\limits_{j=0}^{e-1} \overline{ ap_{2j+1}\gamma_{i-2j-1} \frac{\dd t_I}{t_I}}} &  \textcolor{black}{\text{if}}\; d=2e,\\
    {\sum\limits_{j=0}^{e}\overline{ap_{2j} \gamma_{i-2j} \frac{\dd t_I}{t_I}}} & \textcolor{black}{\text{if}}\; d=2e+1.
            \end{cases} \\&= s_{\frac{1}{x}}^*(\partial_{\frac{1}{x}}(\phi))\quad \textcolor{black}{(\text{by Lemma  \ref{L43}(i)})}.\end{align*}
          Let $\phi$ be a   generator of $ {L_0\wedge \overline{\frac{\dd p}{p}}}$ of second type, i.e,  $ \phi={\overline{ bx^j\frac{\dd t_J}{t_J}\wedge\frac{\dd x}{x}\wedge \frac{\dd p}{p}}},$ where $b\in \F$, $j\geq 1$ and $J\in (T)_{m-2}$ such that  $\tilde{J}=\{l_1,\ldots,l_{m-2}\}$ with $ l_1<l_2<\ldots<l_{m-2}.$ 
        Here, we have two cases depending on $\deg p=d> 1$ or $d=1$.\\

            \noindent{(a)} Let $\deg p=d> 1$. Then, {$p\neq x$, and} we obtain:\begin{align*}
                 s_p^*(\partial_p(\phi))&= s_p^*\left( \overline{ bx^j\frac{\dd t_J}{t_J}\wedge\frac{\dd x}{x}\wedge  \frac{\dd p}{p}}\right)\\&=t_p^*\left( \overline{ bx^j\frac{\dd t_J}{t_J}\wedge \frac{\dd x}{x}}\right)\\&= {f(\F)}^{-1}_m\left(t_p''\left(\overline{ \langle \langle t_{l_1},t_{l_2},\ldots,t_{l_{m-2}},x;bx^j]]}\right)\right)\\&=  {f(\F)}^{-1}_m\left( \overline{ t_p'\left(\langle \langle t_{l_1},t_{l_2},\ldots,t_{l_{m-2}},x;bx^j]]\right)}\right)\\&=  {f(\F)}^{-1}_m\left( \overline{ \left(\langle \langle t_{l_1},t_{l_2},\ldots,t_{l_{m-2}}\rangle\rangle\right)\otimes t_p'(\langle\langle x;bx^j]])}\right)\\&= f(\F)^{-1}_m\left( \sum\limits_{i=1}^{d}\overline{\langle\langle t_{l_1},t_{l_2}, \ldots, t_{l_{m-2}}, p_i; bp_i\gamma_{j-i}]]}\right)\quad \textcolor{black}{(\text{by  Lemma \ref{L42}(ii)})}\\&= \sum\limits_{i=1}^{d} \overline{ bp_i\gamma_{j-i} \frac{\dd t_J}{t_J}\wedge \frac{\dd p_i}{p_i}} = s^*_{\frac{1}{x}}(\partial_{\frac{1}{x}}(\phi))\quad \textcolor{black}{(\text{by Lemma \ref{L43}(ii)})}.
            \end{align*} 
 \noindent{(b)} Let $\deg p=1$. Then $p=x+c$ for some $c\in \F$, and $p_0=1$, $p_1=c$. 
              Now, we have \begin{align*}
                  \phi= \overline{bx^j\frac{\dd t_J}{t_J}\wedge \frac{\dd x}{x}\wedge \frac{\dd p}{p}}&=\overline{bx^j\frac{\dd t_J}{t_J}\wedge \frac{\dd {(p+c)}}{(p+c)}\wedge \frac{\dd p}{p}}\\&= \overline{ \frac{pbx^j}{x}\frac{\dd t_J}{t_J}\wedge \frac{\dd p}{p}\wedge \frac{\dd p}{p}}+\overline{\frac{cbx^j}{x} \frac{\dd t_J}{t_J}\wedge \frac{\dd c}{c}\wedge\frac{\dd p}{p}} \quad \textcolor{black}{(\text{by  (\ref{feq}))}}\\&= \overline{{cbx^{j-1}} \frac{\dd t_J}{t_J}\wedge \frac{\dd c}{c}\wedge\frac{\dd p}{p}} \\&= \overline{{cb(p+c)^{j-1}} \frac{\dd t_J}{t_J}\wedge \frac{\dd c}{c}\wedge\frac{\dd p}{p}}.
              \end{align*}Using Remark \ref{R21} and   definition of the transfer map $s_p^*$, we have  \begin{align*}
                  s_p^*(\partial_p(\phi))= \overline{ {bc^j} \frac{\dd t_J}{t_J}\wedge \frac{\dd c}{c}}&= \overline{ bp_1\gamma_{j-1} \frac{\dd t_J}{t_J}\wedge \frac{\dd p_1}{p_1}}\quad \textcolor{black}{(\text{since } p_1=c\text{ and } \gamma_{j-1}= c^{j-1} )}\\&=s_{\frac{1}{x}}^*(\partial_{\frac{1}{x}}(\phi)) \quad \textcolor{black}{(\text{by Lemma \ref{L43}(ii)})}.
              \end{align*}Finally, for  ${\phi= \overline{bx^j\frac{\dd t_J}{t_J}\wedge \frac{\dd x}{x}\wedge \frac{\dd p}{p}}}$  with $ b\in \F,j\geq 1$ and $J\in (T)_{m-2}$, we have  $ \sum\limits_{q}s_q^*(\partial_q(\phi))=0.$ Taking all together,  we get that $ \sum\limits_{q}s_q^*(\partial_q(\phi))=0$ for all $\phi\in L_0+L_0\wedge \overline{\frac{\dd p}{p}}$ and $p$ monic irreducible.\\

\noindent{(ii)}  Let $p=\sum\limits_{i=0}^{d} p_ix^{d-i}$ be inseparable, where   $p_i\in \F$ and $p_0=1$. {Let $\phi={ \overline{ ax^i\frac{\dd t_I}{t_I}\wedge \frac{\dd x}{x}\wedge \frac{\dd p}{p}}}$, where $a\in \F$, $i\geq 0$ and  $I\in (T)_{m-2}$, a generator of the group ${L_0\wedge \overline{ \frac{\dd x}{x}\wedge \frac{\dd p}{p}}}$.}
Now if $i\geq 1$, then $ {\overline{ax^i\frac{\dd t_I}{t_I}\wedge \frac{\dd x}{x}}}\in L_0$ and it implies that ${\phi\in L_0\wedge \overline{  \frac{\dd p}{p}}}$. Then, by statement  (i), we have $\sum\limits_{q}s^*_q(\partial_q( \phi))=0$. Therefore, {it remains to prove that $\sum\limits_{q}s^*_q(\partial_q( \phi))=0$ for $\phi={ \overline{ a\frac{\dd t_I}{t_I}\wedge \frac{\dd x}{x}\wedge \frac{\dd p}{p}}}$, where   $a\in \F$.} By \cite[Lemma 3.3]{ex},  we have  $ \partial_q\left({ \overline{ a\frac{\dd t_I}{t_I}\wedge \frac{\dd x}{x}\wedge \frac{\dd p}{p}}}\right)=0$  for all $q$ monic irreducible except $p$, $x$ and $\frac{1}{x}.$
 Now we develop the expression of $\phi$:
\begin{align*}
      \phi&= \overline{ a\frac{\dd t_I}{t_I}\wedge \frac{\dd x}{x}\wedge \frac{\dd p}{p}} \\&= \sum_{i=0}^{d}\overline{\frac{ap_ix^{d-i}}{p}\frac{\dd t_I}{t_I}\wedge \frac{\dd x}{x}\wedge \frac{\dd (p_ix^{d-i})}{p_ix^{d-i}} }\quad (\textcolor{black}{\text{by  (\ref{feq})}})\\&= \sum_{i=0}^{d}\overline{\frac{ap_ix^{d-i}}{p}\frac{\dd t_I}{t_I}\wedge \frac{\dd x}{x}\wedge \frac{\dd p_i}{p_i}} \\&= \sum_{i=1}^{d}\overline{\frac{ap_ix^{d-i}}{p}\frac{\dd t_I}{t_I}\wedge \frac{\dd p_i}{p_i}\wedge \frac{\dd x}{x}} \quad \textcolor{black}{(\text{since }   p_0=1)}.
  \end{align*}
Since  $ \deg(ap_ix^{d-i})<\deg p$, so ${v_{\frac{1}{x}}\left( \frac{ap_ix^{d-i}}{p}\right)}\geq 1 $ for all $1\leq i\leq d.$
   Therefore, $ \partial_{\frac{1}{x}}(\phi)=0.$
So we need to prove that $$s_x^*(\partial_x(\phi))= s_p^*(\partial_p(\phi)).$$ If $1\leq i<d$, then $ {v_{{x}}\left( \frac{ap_ix^{d-i}}{p}\right)}\geq 1 $.   \textcolor{black}{ Since $p_d$ is the constant term of $p$,  it follows that  ${\frac{ap_d}{p}}= a+u\in \F((x))$, where $v_x(u)\geq 1$.} Therefore,  we have  \begin{align*} \partial_x(\phi)= \partial_x\left( \overline{\frac{ap_d}{p} \frac{\dd t_I}{t_I} \wedge \frac{\dd p_d}{p_d}\wedge \frac{\dd x}{x}}\right)= \overline{ a \frac{\dd t_I}{t_I}\wedge \frac{\dd p_d}{p_d}\wedge \frac{\dd x}{x}}.\end{align*}  Consequently, by definition of the transfer map $s_x^*$, we have  \begin{equation}\label{xeq} {s_x^*(\partial_x(\phi))=\overline{ a \frac{\dd t_I}{t_I}\wedge \frac{\dd p_d}{p_d}}}.\end{equation}
      Now we compute $s_p^*(\partial_p(\phi))$. Let $\tilde{I}=\{i_1,i_2,\ldots,i_{m-2}\}$, with $i_1<i_2<\ldots<i_{m-2}$. Then, we have   

\begin{align*} s_p^*(\partial_p( \phi))&= s_p^*\left(\partial_p\left( \overline{ a\frac{\dd t_I}{t_I}\wedge \frac{\dd x}{x}\wedge \frac{\dd p}{p}}\right)\right)\\&=t_p^*\left(\overline{ a \frac{\dd t_I}{t_I}\wedge \frac{\dd x}{x}}\right) \quad \textcolor{black}{(\text{since } \deg p>1)} \\&= f(\F)^{-1}_m\left( t_p''\left( \overline{ \langle\langle t_{i_1},\ldots, t_{i_{m-2}},x;a]}\right)\right)\\&= f(\F)^{-1}_m\left( \overline{ t_p'\left( \langle\langle t_{i_1},\ldots,t_{i_{m-2}},x;a]]\right)}\right)\\&= f(\F)^{-1}_m\left( \overline{  \langle\langle t_{i_1},\ldots, t_{i_{m-2}}\rangle\rangle\otimes t_p'\left( \langle\langle x;a]]\right)}\right)\\&= f(\F)^{-1}_m\left( \overline{\langle\langle t_{i_1},\ldots, t_{i_{m-2}}\rangle\rangle\otimes (\langle\langle p_d;a]])}\right)\quad \textcolor{black}{(\text{by Lemma \ref{L42}(iii)})}\\&=\overline{a\frac{\dd t_I}{t_I}\wedge \frac{\dd p_d}{p_d}}=s_x^*(\partial_x(\phi))\quad (\textcolor{black}{\text{by (\ref{xeq})}}).\end{align*}
  Thus,  ${\sum\limits_{q}s_q^*(\partial_q(\phi))=0}$ for all $\phi\in {L_0\wedge \overline{\frac{\dd x}{x}\wedge \frac{\dd p}{p}}}$, where $p$ is  an   inseparable monic irreducible polynomial.
    \end{proof}

\section{Milnor-Scharlau's exact sequence}\label{S5}
Our aim in this section is to extend the reciprocity law for   $ H_2^{m+1}(\F(x))$, where $\F$ is a field of characteristic $2$ and prove the main result of the paper.

\begin{thm}\label{T51} For any $\phi\in H_2^{m+1}(\F(x))$, we have  $\sum\limits_{q} s_q^*(\partial_q(\phi))=0$, where $q$ varies over  $\frac{1}{x}$ and all monic irreducible polynomials over $\F$.
 \end{thm}
 \begin{proof}
     By \cite[Lemma 3.5(2)]{ex}, we have $H_2^{m+1}(\F(x))=\sum\limits_{p}\left(S_p+S_p\wedge \overline{\frac{\dd x}{x}}\right),$ where  $p$ varies over all  monic irreducible polynomials. 
\medskip

-- If $p$ is separable, then   Lemma \ref{L32}(ii) implies that $$S_p+S_p\wedge \overline{\frac{\dd x}{x}}\subseteq U_p+L_0+L_0\wedge \overline{\frac{\dd x}{x}}+L_0\wedge \overline{\frac{\dd p}{p}}.$$By Lemma \ref{L44}(i), we have      $\sum\limits_q s_q^*(\partial_q(\phi))=0$ for all $\phi\in L_0+L_0\wedge \overline{\frac{\dd x}{x}}+L_0\wedge \overline{\frac{\dd p}{p}}$. 
\medskip

-- If $p$ is  inseparable,  then by  Lemma \ref{L32}(iv)  we have $$S_p+S_p\wedge \overline{\frac{\dd x}{x}}\subseteq  U_p+L_0+L_0\wedge \overline{\frac{\dd x}{x}}+ L_0\wedge \overline{ \frac{\dd p}{ p}}+L_0\wedge \overline{\frac{\dd x}{x}\wedge \frac{\dd p}{p}}. $$Moreover, for all ${\phi\in L_0+L_0\wedge \overline{\frac{\dd x}{x}}+ L_0\wedge \overline{ \frac{\dd p}{ p}}+L_0\wedge \overline{\frac{\dd x}{x}\wedge \frac{\dd p}{p}}}$, we have by Lemma \ref{L44} ${\sum\limits_{q}s_q^*(\partial_q(\phi))=0}$.

     By Definition (\ref{Up}), $U_p=\bigoplus\limits_{r\geq 1}S'_{p,r}$ for all $p$ monic  irreducible. 
Therefore, it is sufficient to prove that  ${\sum\limits_{q} s_q^*(\partial_q(\phi))=0}$, where $\phi$ is a generator of   $ S'_{p,r}$ for any  monic irreducible $p$ and  $r\geq 1.$ We  prove this in four cases.\\

    \noindent{(i)} Let ${\phi=\overline{\frac{\overline{t^J s^2_{r,I,J}}}{p^r}\frac{\dd t_I}{t_I}}}$, where  $ J\in \overline{T}_p,r\geq 1 $,  $s_{r,I,J}\in \F[x]_{< \deg p}$ and  $I\in ({T}_p)_{m}$ such that $t^I\in \F$. Moreover,  $J+I> I$ when $r$ is even.
 By \cite[Lemma 3.3]{ex}, we get  $\partial_q(\phi)=0$ for all $q$ 
 except $p$ and $\frac{1}{x}$.
Since $\deg \left(\overline{t^J s^2_{r,I,J}}\right)<\deg p$, it follows that  $v_{\frac{1}{x}}\left(\frac{\overline{t^J s^2_{r,I,J}}}{p^r}\right)>0$. Hence ${\partial_{\frac{1}{x}}(\phi)=0}.$
 Now  the  definition of $s_p^*$ and Theorem \ref{T23} implies that   $ s_p^*(\partial_p(\phi))=0.$  Because $\partial_p(\phi)\in \overline{U_p}$ for $p$ is separable, and $\partial_p(\phi)\in \overline{U_p^0}$ for $p$ inseparable. Thus, for this particular type of generators $\phi$ of $S'_{p,r}$, we have $\sum_{q} s_q^*(\partial_q(\phi))=0.$

\noindent{(ii)} Let ${\phi=\overline{\frac{\overline{t^J s^2_{r,I,J}}}{p^r}\frac{\dd t_I}{t_I}\wedge \frac{\dd p}{p}}}, $ where  $r$ is even, $ s_{r,I,J}\in \F[x]_{< \deg p}$ and $I\in ({T}_p)_{m-1}$ such that $t^I\in \F$. Moreover, $J\in \overline{T}_p\subset T_p$ such that $J+I> I$.   Now  \cite[Lemma 3.3]{ex} implies that   $\partial_q(\phi)=0$ for all ${q\neq p,\frac{1}{x}}.$ By the same reasons as in (i), we also have $ \partial_{\frac{1}{x}}(\phi)=0$ and $s_p^*(\partial_p(\phi))=0$.

The elements {considered} in (i) and (ii) generate $U_p$ for $p$ separable and $U_p^0$ for $p$ inseparable.  It remains  to check for the generators of $U_p\setminus U_p^0$, when $p$ is  inseparable.  Therefore, for  the next two  types  of generators, let $p$ be inseparable.
\vskip1mm

\noindent{(iii)} Let $ {\phi=\overline{\frac{\overline{t^J s^2_{r,I,J}}}{p^r}\frac{\dd t_I}{t_I}\wedge \frac{\dd x}{x}}},$ where $J\in \overline{T}_p, r\geq 1$, $s_{r,I,J}\in \F[x]_{< \deg p}$ and  $I\in ({T}_p)_{m-1}$ such that $t^I\in \F$.
      Moreover for $r$ even, $J+I'> I'$, where $t^{I'}=t^I x$. Now \cite[Lemma 3.3]{ex} implies that   $\partial_q(\phi)=0$ for all $q$ except $ p,x$ and $\frac{1}{x}$.  Since $\deg\left(\overline{t^Js^2_{r,I,J}}\right)<\deg p$,   Remark \ref{R21} implies that $\partial_{\frac{1}{x}}(\phi)=0$. Therefore, it suffices  to  show that $ s_p^*(\partial_p(\phi))=s_x^*(\partial_x(\phi))$. Now    we write $\frac{\overline{t^J s^2_{r,I,J}}}{p^r}=\frac{\left(\overline{t^J s^2_{r,I,J}}\right)_c}{p_c^r}+ u\in \F((x))$, where $v_x(u)\geq 1$
 and  $f_c$ denotes the constant term
  of $f$, for any polynomial $f\in \F[x]$.
Therefore, we have \begin{align*} s_x^*\left(\partial_x(\phi)\right)&= s_x^*\left( \overline{ \frac{\left(\overline{t^J s^2_{r,I,J}}\right)_c}{p_c^r}\frac{\dd t_I}{t_I}\wedge \frac{\dd x}{x}}\right)= \overline{ \frac{\left(\overline{t^J s^2_{r,I,J}}\right)_c}{p_c^r}\frac{\dd t_I}{t_I}}. \end{align*} Using the decomposition of Theorem \ref{T23}, ${\partial_p(\phi)=\overline{\frac{\overline{t^J s^2_{r,I,J}}}{p^r}\frac{\dd t_I}{t_I}\wedge \frac{\dd x}{x}}}$. Since $p$ is inseparable, it folllows that   $$ {s_p^*( \partial_p(\phi))=\overline{ \frac{\left(\overline{t^J s^2_{r,I,J}}\right)_c}{p_c^r}\frac{\dd t_I}{t_I}}} = s_x^*(\partial_x(\phi)).$$\\

 \noindent{(iv)} Let $ {\phi=\overline{\frac{\overline{t^Js^2_{r,I,J}}}{p^r}\frac{\dd t_I}{t_I}\wedge \frac{\dd x}{x}\wedge \frac{\dd p}{p}}},$ where  $r$ is even, $ s_{r,I,J}\in \F[x]_{< \deg p}$ and $I\in ({T}_p)_{m-2}$ such that $t^I\in \F$. Moreover, $J\in \overline{T}_p$ such that $J+I'> I'$, where $t^{I'}=t^Ix$. As in  part (iii), here also $ \partial_q(\phi)=0$ for all $q$ except $ p$ and $x.$ Hence we need to prove that $ s_p^*(\partial_p(\phi))=s_x^*(\partial_x(\phi))$. Now by  the definition of $s_p^*$, we have $${s_p^*(\partial_p(\phi))=\overline{\frac{\left(\overline{t^J s^2_{r,I,J}}\right)_c}{p_c^r}\frac{\dd t_I}{t_I}\wedge \frac{\dd p_c}{p_c}}}.$$

Let $p=\sum\limits_{i=0}^{d}p_ix^{d-i}$ with $p_0=1$. Let $h$ denote $\overline{t^J s^2_{r,I,J}}$, which is a polynomial of degree $<d$.   Using (\ref{feq}) we write   \begin{align*}
    \phi&= \sum_{i=0}^{d} \overline{ \frac{hp_ix^{d-i}}{p^{r+1}}\frac{\dd t_I}{t_I}\wedge \frac{\dd x}{x}\wedge \frac{\dd p_i}{p_i}}\\&= \sum_{i=1}^{d} \overline{ \frac{hp_ix^{d-i}}{p^{r+1}}\frac{\dd t_I}{t_I}\wedge \frac{\dd p_i}{p_i}\wedge \frac{\dd x}{x}} \quad\textcolor{black}{(\text{as }  p_0=1)}\end{align*}
Hence, we get
\begin{align*}
\partial_x(\phi)&= \overline{ \frac{h_c p_d}{p^{r+1}_c}\frac{\dd t_I}{t_I}\wedge \frac{\dd p_d}{p_d}\wedge \frac{\dd x}{x}}\quad \textcolor{black}{(\text{since  } v_x\left( \frac{hp_ix^{d-i}}{p}\right)\geq 1 \text{ for } \hspace{1mm} i<d )} \\&=  \overline{ \frac{h_c }{p^{r}_c}\frac{\dd t_I}{t_I}\wedge \frac{\dd p_c}{p_c}\wedge \frac{\dd x}{x}} \quad \textcolor{black}{(\text{since } p_c=p_d)}.\end{align*}
    Therefore, \begin{align*}  s_x^*(\partial_x(\phi))= \overline{ \frac{h_c }{p^{r}_c}\frac{\dd t_I}{t_I}\wedge \frac{\dd p_c}{p_c}}= \overline{ \frac{\left( \overline{t^Js^2_{r,I,J}}\right)_c }{p^{r}_c}\frac{\dd t_I}{t_I}\wedge \frac{\dd p_c}{p_c}}= s_p^*(\partial_p(\phi)).
\end{align*}  
Taking (i), (ii), (iii) and (iv)  together, we get that $\sum\limits_{q}(s_q^*(\partial_q(\phi))=0$ for all $\phi\in U_p$, for any  monic  irreducible $p$, which proves the theorem.  
 \end{proof}

Now we are able to prove our main result, the Milnor-Scharlau exact sequence in the setting of Kato-Milne cohomology.

  \begin{thm}\label{milnor}
       Let $\mathcal{F}$ be a field of characteristic $2$  and $F=\mathcal{F}(x)$,  the rational functional field in one variable $x$ over $\mathcal{F}$. Then the following sequence is   exact:\[ 0\xrightarrow{} H_2^{m+1}(\mathcal{F})\xrightarrow{{i}} H_2^{m+1}(F)\xrightarrow{{\bigoplus\partial_p}}{ \bigoplus_{p,\frac{1}{x}}W_1(H_2^{m+1}(F_p))} \xrightarrow{\bigoplus s_p^*} H_2^m(\F)\xrightarrow{}0, \] where ${i}$ is induced by the inclusion  $\F\xrightarrow{}F$, and  $p$ varies over   $\frac{1}{x} $ and   all monic  irreducible polynomials  of $ \F[x]$.  
  \end{thm}

  \begin{proof}
       Now \cite[Lemma 2.17]{AB} provides the injectivity of ${i}$. By the exactly same method of \cite[Theorem 4.10]{ex}, we can  prove that  $ \textcolor{black}{\rm ker}(\bigoplus \partial_p) ={\rm Im}({i})$.  Moreover,  \cite[Theorem 4.9]{ex} implies that      $s_{\frac{1}{x}}^*=\eta$ is onto ($\eta$ as in \cite[Theorem 4.9]{ex}). Consequently, the map $\bigoplus s_p^*$ is onto.   Now  because of {Theorem \ref{T51}}, it suffices to  prove that  \textcolor{black}{${\rm ker}\left(\bigoplus s_p^*\right) \subseteq {\rm Im}\left(\bigoplus\partial_p\right)$}.\vskip1mm

{Let ${\mathcal P}$ be the set of all irreducible monic polynomials of $\F[x]$.} Let us consider $ {\bigoplus s_p^*( \phi_1,\phi_2)=0}$, where  $\phi_1\in \bigoplus\limits_{{p\in {\mathcal P}}} W_1(H_2^{m+1}(F_p))$ and $\phi_2\in  W_1(H_2^{m+1}(F_{\frac{1}{x}})).$ 
 Therefore, we have \begin{equation}\label{M1}
        \bigoplus\limits_{{p\in {\mathcal P}}} s_p^*( \phi_1)\bigoplus s^*_{\frac{1}{x}}(\phi_2)=0.
\end{equation}
By \cite[Theorem 4.8]{ex}, there exists an element $\phi\in H_2^{m+1}(F)$ such that $\bigoplus\limits_{{p\in {\mathcal P}}} \partial_p(\phi)=\phi_1.$  
  Putting this expression of $\phi_1$ in (\ref{M1}), we get:$$\bigoplus_{{p\in {\mathcal P}}} 
s_p^*(\partial_p(\phi))\bigoplus s^*_{\frac{1}{x}}(\phi_2)=0.$$
By Theorem \ref{T51}, we have:$$s_{\frac{1}{x}}^*\left(\partial_{\frac{1}{x}}(\phi)\right)+ s_{\frac{1}{x}}^*(\phi_2)=0.$$
Since $ s_{\frac{1}{x}}^*=\eta$, it follows that$$\eta\left(\partial_{\frac{1}{x}}(\phi)+\phi_2\right)=0.$$

Moreover, \cite[Theorem 4.9]{ex} implies that  $\partial_{\frac{1}{x}}(\phi)+\phi_2= i_2(\phi')$, where $i_2$ is the restriction of $\partial_{\frac{1}{x}}$ on $L_0$, and $\phi'\in L_0$. Therefore, we have $\partial_{\frac{1}{x}}(\phi)+\phi_2=  \partial_{\frac{1}{x}}(\phi')$, and thus we get$$\partial_{\frac{1}{x}}(\phi+ \phi')=\phi_2.$$

Since $\phi'\in L_0$, \cite[Theorem 4.8]{ex} implies that  $ \bigoplus\limits_{{p\in {\mathcal P}}} \partial_p( \phi')=0.$  Hence, $ \bigoplus \partial_p( \phi+\phi')=(\phi_1,\phi_2).$   Therefore, ${\rm ker}\left(\bigoplus s_p^*\right) \subseteq {\rm Im} \left(\bigoplus\partial_p\right)$.
 Thus, the theorem is proved.
  \end{proof}

\end{sloppypar}
 \end{document}